\documentclass[12pt,english]{amsart}
\usepackage[T1]{fontenc}
\usepackage[latin9]{inputenc}
\usepackage[letterpaper]{geometry}
\usepackage{amsthm}
\usepackage{amstext}
\usepackage{amssymb}
\usepackage{esint}

\makeatletter
\numberwithin{equation}{section} 
\numberwithin{figure}{section} 
\theoremstyle{plain}
\newtheorem{thm}{Theorem}
  \theoremstyle{plain}
  \newtheorem{lem}[thm]{Lemma}
  \theoremstyle{plain}
  \newtheorem{cor}[thm]{Corollary}
  \theoremstyle{plain}
  \newtheorem{prop}[thm]{Proposition}

\newcommand{\fdim}{\operatorname{fdim}}
\author[A. Guionnet, V. Jones, D. Shlyakhtenko]{A. Guionnet$^*$, V. Jones$^\dagger$ and D. Shlyakhtenko$^\ddagger$}
\thanks{}

\makeatother

\usepackage{babel}

\begin{document}

\title{A semi-finite algebra associated to a subfactor planar algebra.}

\thanks{$^{*}$UMPA, ENS Lyon, 46 alée d'Italie, 69364 Lyon Cedex 07, France,
aguionne@umpa.ens-lyon.fr. Research supported by ANR project ANR-08-BLAN-0311-01.
\\
$^{\dagger}$Department of Mathematics, UC Berkeley, Berkeley,
CA 94720, vfr@math.berkeley.edu. Research supported by NSF grant DMS-0856316.\\
$^{\ddagger}$Department of Mathematics, UCLA, Los Angeles, CA
90095 USA, shlyakht@math.ucla.edu. Reserach supported by NSF grants
DMS-0555680, DMS-0900776.}
\begin{abstract}
We canonically associate to any planar algebra two type II$_{\infty}$
factors $\mathfrak{M}_{\pm}$. The subfactors constructed previously
by the authors in \cite{guionnet-jones-shlyakhtenko1} are isomorphic
to compressions of $\mathfrak{M}_{\pm}$ to finite projections. We
show that each $\mathfrak{M}_{\pm}$ is isomorphic to an amalgamated
free product of type I von Neumann algebras with amalgamation over
a fixed discrete type I von Neumann subalgebra. In the finite-depth
case, existing results in the literature imply that $\mathfrak{M}_{+}\cong\mathfrak{M}_{-}$
is the amplification a free group factor on a finite number of generators.
As an application, we show that the factors $M_{j}$ constructed in
\cite{guionnet-jones-shlyakhtenko1} are isomorphic to interpolated
free group factors $L(\mathbb{F}(r_{j}))$, $r_{j}=1+2\delta^{-2j}(\delta-1)I$,
where $\delta^{2}$ is the index of the planar algebra and $I$ is
its global index. Other applications include computations of laws
of Jones-Wenzl projections.
\end{abstract}
\maketitle

\section{Introduction.}

In this paper, we associate a pair of semi-finite von Neumann algebras
$\mathfrak{M}_{\pm}$ to a planar algebra $\mathcal{P}$. The algebras
$\mathfrak{M}_{\pm}$ are obtained via the GNS construction from a
certain non-unital tracial inductive limit algebra $V_{+}$ which
arises canonically from $\mathcal{P}$. These algebras have an interesting
structure, and the paper is mainly devoted to their study.

To state our main application, let $\mathcal{P}$ be a subfactor planar
algebra of index $\delta^{2}$, and let us denote by $M_{k}=M_{k}(\mathcal{P})$
the von Neumann algebra generated in the GNS representation of $(\mathcal{P},\wedge_{k},Tr_{k})$
(see Def. 7 and 8 in \cite{guionnet-jones-shlyakhtenko1}). We prove:
\begin{thm}
\label{thm:ClassofMk}Assume that $\mathcal{P}$ is finite-depth with
global index $I$. Then $M_{k}\cong L(\mathbb{F}(r_{k}))$ with $r_{k}=1+2\delta^{-2k}(\delta-1)I$.
\end{thm}
We refer the reader to \cite{ocneanu:globalIndex,evans-kawahighashi}
for the definition of global index $I$. If $\Gamma$ is the principal
graph of $\mathcal{P}$ and $\mu$ is the Perron-Frobenius eigenvector
normalized by $\mu(*)=1$, then $I=\frac{1}{2}\sum_{v\in\Gamma}\mu(v)^{2}$.
This formula is consistent with the result of Kodiyalam and Sunder
in the depth two case \cite{kodiyalam-sunder:depth2}. 

The main step in proving \ref{thm:ClassofMk} is to prove that the
amplifications of $M_{k}$ are isomorphic to type II$_{\infty}$ von
Neumann algebra $\mathfrak{M}_{+}$ or $\mathfrak{M}_{-}$ (the choice
of sign is according to the parity of $k$). It turns out that each
$\mathfrak{M}_{\pm}$ admits a description as a (possibly infinite)
free product with amalgamation over a discrete type I von Neumann
subalgebra of type I von Neumann algebras. In the finite-depth case,
this is sufficient to determine the isomorphism class of $\mathfrak{M}_{\pm}$
using the work of Dykema \cite{brown-dykema-jung:fdimamalg,dykema:fdim,dykema:interpolated,dykema:finiteAlgAmalg}.

We note that $r_{k}$ the statement of Theorem \ref{thm:ClassofMk}
satisfy $(r_{k}-1)=\delta^{2}(r_{k+1}-1)$. 

We mention also that while this paper was in preparation, Kodiyalam
and Sunder have found a different proof that $M_{k}$ are (in the finite-depth
case) isomorphic to interpolated free group factors \cite{kodiyalam-sunder:interpolated}. 

We conclude the paper with another application of the isomorphism
between $M_{0}$ and a compression of $\mathfrak{M}$. It allows us
to recover the random matrix model used in \cite{guionnet-jones-shlyakhtenko1}
and can be quite useful in random matrix computations (we illustrate
this by describing the joint law of Jones-Wenzl idempotents $JW$).

\section{A semi-finite tracial algebra associated to a planar algebra.}

Let $\mathcal{P}$ be a planar algebra. We denote by $\mathcal{P}_{k}^{\epsilon}$,
$k=0,1,2,\dots$, $\epsilon=\pm$, the $k$-th graded component of
$\mathcal{P}$. 

For fixed $k,\epsilon$ and integers $a,b,p$ satisfying $a+b+p=2k$,
let $V_{a,b}^{\epsilon}(p)$ be a copy of $\mathcal{P}_{k}^{\epsilon}$;
we think of $V_{a,b}^{\epsilon}(p)$ as diagrams arranged to have,
clockwise from the first string, $a$ strings on the left, $p$ on
top, $b$ on the right and so that the top-left corner has shading
$\epsilon$. Define the multiplication map\[
\cdot:V_{a,b}^{\epsilon}(p)\times V_{a',b'}^{\epsilon'}(p')\to V_{a,b'}^{\epsilon}(p+p')\]
 to be zero unless $b=a'$ and $\epsilon'=(-1)^{p}\epsilon$ and otherwise
by the tangle:\def\figureone{\unitlength 2.0 mm
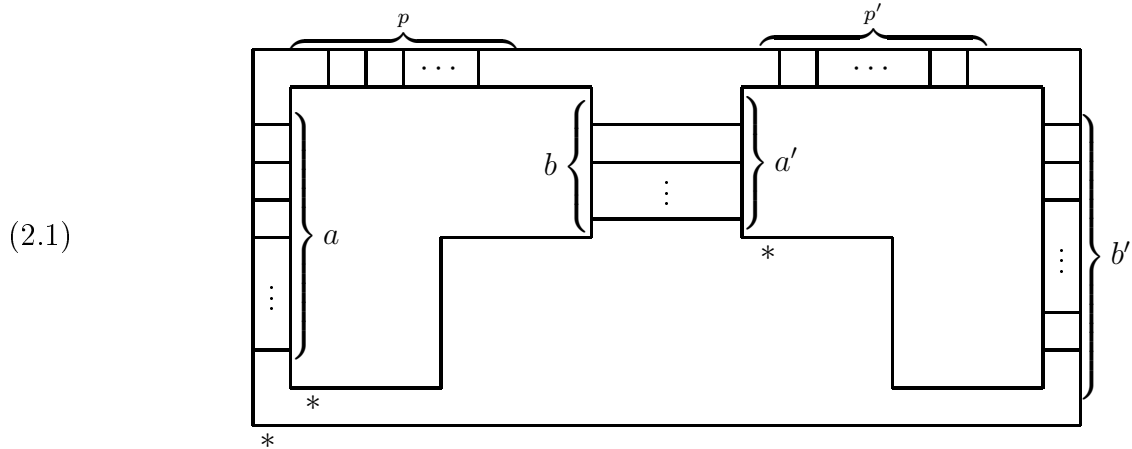
\begin{picture}(65,27)(0,37) \linethickness{0.3mm} \put(10,40){\line(0,1){20}} \linethickness{0.3mm} \put(20,40){\line(0,1){10}} \linethickness{0.3mm} \put(10,40){\line(1,0){10}} \linethickness{0.3mm} \put(20,50){\line(1,0){10}} \linethickness{0.3mm} \put(30,50){\line(0,1){10}} \linethickness{0.3mm} \put(10,60){\line(1,0){20}} \linethickness{0.3mm} \put(40,50){\line(0,1){10}} \linethickness{0.3mm} \put(40,50){\line(1,0){10}} \linethickness{0.3mm} \put(50,40){\line(0,1){10}} \linethickness{0.3mm} \put(50,40){\line(1,0){10}} \linethickness{0.3mm} \put(60,40){\line(0,1){20}} \linethickness{0.3mm} \put(40,60){\line(1,0){20}} \linethickness{0.3mm} \put(30,55){\line(1,0){10}} \linethickness{0.3mm} \put(30,57.5){\line(1,0){10}} \linethickness{0.3mm} \put(30,51.25){\line(1,0){10}} \linethickness{0.3mm} \put(12.5,60){\line(0,1){2.5}} \linethickness{0.3mm} \put(15,60){\line(0,1){2.5}} \linethickness{0.3mm} \put(17.5,60){\line(0,1){2.5}} \linethickness{0.3mm} \put(22.5,60){\line(0,1){2.5}} \linethickness{0.3mm}
\put(42.5,60){\line(0,1){2.5}} \put(42.5,60){\line(0,1){2.5}}
\linethickness{0.3mm}
\put(45,60){\line(0,1){2.5}} \put(45,60){\line(0,1){2.5}}
\linethickness{0.3mm}
\put(52.5,60){\line(0,1){2.5}} \put(52.5,60){\line(0,1){2.5}}
\linethickness{0.3mm}
\put(55,60){\line(0,1){2.5}} \put(55,60){\line(0,1){2.5}}
\linethickness{0.3mm} \put(60,57.5){\line(1,0){2.5}}
\put(60,57.5){\line(1,0){2.5}} \linethickness{0.3mm} \put(60,55){\line(1,0){2.5}}
\put(60,55){\line(1,0){2.5}} \linethickness{0.3mm} \put(60,52.5){\line(1,0){2.5}}
\put(60,52.5){\line(1,0){2.5}} \linethickness{0.3mm} \put(60,42.5){\line(1,0){2.5}}
\put(60,42.5){\line(1,0){2.5}} \linethickness{0.3mm} \put(60,45){\line(1,0){2.5}}
\put(60,45){\line(1,0){2.5}} \linethickness{0.3mm} \put(7.5,57.5){\line(1,0){2.5}} \linethickness{0.3mm} \put(7.5,55){\line(1,0){2.5}} \linethickness{0.3mm} \put(7.5,52.5){\line(1,0){2.5}} \linethickness{0.3mm} \put(7.5,50){\line(1,0){2.5}} \linethickness{0.3mm} \put(7.5,42.5){\line(1,0){2.5}} \linethickness{0.3mm} \put(7.5,62.5){\line(1,0){55}} \put(7.5,37.5){\line(0,1){25}} \put(62.5,37.5){\line(0,1){25}} \put(7.5,37.5){\line(1,0){55}} \put(8.75,46.5){\makebox(0,0)[cc]{$\vdots$}}
\put(8.75,45){\makebox(0,0)[cc]{}}
\put(61.25,49){\makebox(0,0)[cc]{$\vdots$}}
\put(65,42.5){\makebox(0,0)[cc]{}}
\put(35,53.5){\makebox(0,0)[cc]{$\vdots$}}
\put(20,61.25){\makebox(0,0)[cc]{$\cdots$}}
\put(48.75,61.25){\makebox(0,0)[cc]{$\cdots$}}
\put(17.5,63.75){\makebox(0,0)[cc]{{$\overbrace{\hbox to 3cm{\hfill}}^{p}$}}}
\put(13.75,63.75){\makebox(0,0)[cc]{}}
\put(17.5,65){\makebox(0,0)[cc]{}}
\put(48.75,64.25){\makebox(0,0)[cc]{$\overbrace{\hbox to 3cm{\hfill}}^{p'}$}}
\put(50,63.75){\makebox(0,0)[cc]{}}
\put(64,48.75){\makebox(0,0)[cc]{$\left.\vbox to 2cm{\vfill}\right\} b'$}}
\put(63.75,48.75){\makebox(0,0)[cc]{}}
\put(63.75,48.75){\makebox(0,0)[cc]{}}
\put(28.5,54.75){\makebox(0,0)[cc]{$b\left\{\vbox to 1cm{\vfill}\right.$}}
\put(28.75,53.75){\makebox(0,0)[cc]{}}
\put(11.5,50.25){\makebox(0,0)[cc]{$\left.\vbox to 1.8cm{\vfill}\right\} a$}}
\put(11.5,39){\makebox(0,0)[cc]{$*$}}
\put(41.75,55){\makebox(0,0)[cc]{$\left.\vbox to 1cm{\vfill}\right\} a'$}}
\put(41.75,49){\makebox(0,0)[cc]{$*$}}
\put(18.75,56.25){\makebox(0,0)[cc]{}}
\put(55,36.25){\makebox(0,0)[cc]{}}
\put(55,35){\makebox(0,0)[cc]{}}
\put(43.75,32.5){\makebox(0,0)[cc]{}}
\put(8.5,36.5){\makebox(0,0)[cc]{$*$}}
\end{picture}}\begin{equation}
\raisebox{-2.5cm}[3cm]{\figureone}\label{eq:multiplication}\end{equation}

The two choices of shading at the top left of the tangle correspond
to the possible values of $\epsilon$.

Define the trace $Tr:V_{a,b}^{\epsilon}(p)\to\mathcal{P}_{0}^{\epsilon}$
to be zero unless $a=b$ and otherwise by the tangle\def\figuretwo{\unitlength 0.7mm 
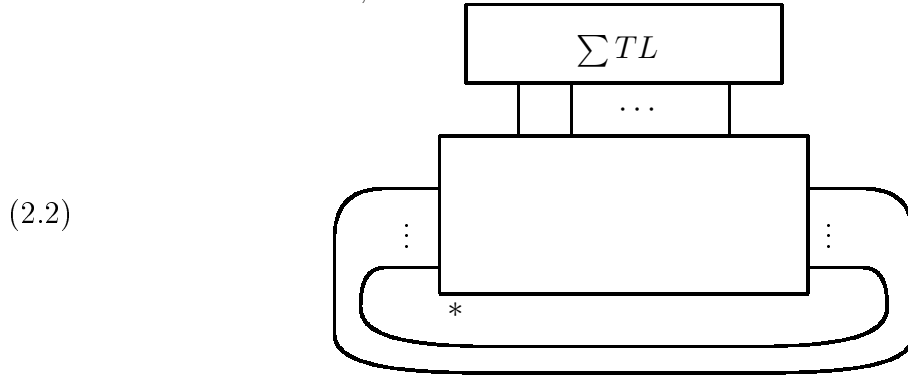
\begin{picture}(110,45)(0,05) \linethickness{0.3mm} \put(20,50){\line(1,0){70}} \put(20,20){\line(0,1){30}} \put(90,20){\line(0,1){30}} \put(20,20){\line(1,0){70}} \linethickness{0.3mm} \put(10,40){\line(1,0){10}} \linethickness{0.3mm} \put(10,25){\line(1,0){10}} \linethickness{0.3mm} \put(90,40){\line(1,0){10}} \linethickness{0.3mm} \put(90,25){\line(1,0){10}} \linethickness{0.3mm} \put(35,50){\line(0,1){10}} \linethickness{0.3mm} \put(45,50){\line(0,1){10}} \linethickness{0.3mm} \put(75,50){\line(0,1){10}} \linethickness{0.3mm} \put(25,75){\line(1,0){60}} \put(25,60){\line(0,1){15}} \put(85,60){\line(0,1){15}} \put(25,60){\line(1,0){60}} \put(53.75,66.25){\makebox(0,0)[cc]{$\sum TL$}}
\put(58,55){\makebox(0,0)[cc]{$\cdots$}}
\put(13.75,32.5){\makebox(0,0)[cc]{$\vdots$}}
\put(93.75,32.5){\makebox(0,0)[cc]{$\vdots$}}
\put(95,33.75){\makebox(0,0)[cc]{}}
\linethickness{0.3mm} \qbezier(10,25)(7.39,25.02)(6.19,23.22) \qbezier(6.19,23.22)(4.98,21.41)(5,17.5) \qbezier(5,17.5)(4.92,13.59)(10.94,11.78) \qbezier(10.94,11.78)(16.95,9.98)(30,10) \qbezier(30,10)(42.97,10)(55,10) \qbezier(55,10)(67.03,10)(80,10) \qbezier(80,10)(93.05,9.98)(99.06,11.78) \qbezier(99.06,11.78)(105.08,13.59)(105,17.5) \qbezier(105,17.5)(105.02,21.41)(103.81,23.22) \qbezier(103.81,23.22)(102.61,25.02)(100,25) \linethickness{0.3mm} \qbezier(10,40)(4.78,40.03)(2.38,37.62) \qbezier(2.38,37.62)(-0.03,35.22)(0,30) \qbezier(0,30)(0,24.8)(0,20.59) \qbezier(0,20.59)(0,16.38)(0,12.5) \qbezier(0,12.5)(-0.09,8.59)(6.53,6.78) \qbezier(6.53,6.78)(13.15,4.98)(27.5,5) \qbezier(27.5,5)(41.77,5)(55,5) \qbezier(55,5)(68.23,5)(82.5,5) \qbezier(82.5,5)(96.85,4.98)(103.47,6.78) \qbezier(103.47,6.78)(110.09,8.59)(110,12.5) \qbezier(110,12.5)(110,16.38)(110,20.59) \qbezier(110,20.59)(110,24.8)(110,30) \qbezier(110,30)(110.03,35.22)(107.62,37.62) \qbezier(107.62,37.62)(105.22,40.03)(100,40) 
\put(23,17){\makebox(0,0)[cc]{$*$}}
\end{picture} }

\begin{equation}
\raisebox{-2cm}[2.3cm]\figuretwo\label{eq:trace}\end{equation}
(here $\sum TL$ denotes the sum of all Temperley-Lieb diagrams).
Finally, consider the inclusions\[
V_{a,b}^{\epsilon}(p)\to V_{a+2r,b+2s}^{\epsilon}(p)\]
 given by the tangle\def\figurethree{\unitlength 0.75mm \begin{picture}(55,75)(0,0) \linethickness{0.3mm} \put(10,70){\line(1,0){40}} \linethickness{0.3mm} \put(50,30){\line(0,1){40}} \linethickness{0.3mm} \put(30,30){\line(1,0){20}} \linethickness{0.3mm} \put(30,30){\line(0,1){20}} \linethickness{0.3mm} \put(10,50){\line(1,0){20}} \linethickness{0.3mm} \put(10,50){\line(0,1){20}} \linethickness{0.3mm} \put(50,65){\line(1,0){10}} \linethickness{0.3mm} \put(50,55){\line(1,0){10}} \linethickness{0.3mm} \put(50,35){\line(1,0){10}} \linethickness{0.3mm} \put(0,65){\line(1,0){10}} \linethickness{0.3mm} \put(0,55){\line(1,0){10}} \linethickness{0.3mm} \qbezier(0,30)(5.22,30.02)(7.62,28.81) \qbezier(7.62,28.81)(10.03,27.61)(10,25) \qbezier(10,25)(10.03,22.39)(7.62,21.19) \qbezier(7.62,21.19)(5.22,19.98)(0,20) \linethickness{0.3mm} \qbezier(0,15)(5.22,15.02)(7.62,13.81) \qbezier(7.62,13.81)(10.03,12.61)(10,10) \qbezier(10,10)(10.03,7.39)(7.62,6.19) \qbezier(7.62,6.19)(5.22,4.98)(0,5) \linethickness{0.3mm} \qbezier(0,45)(5.22,45.02)(7.62,43.81) \qbezier(7.62,43.81)(10.03,42.61)(10,40) \qbezier(10,40)(10.03,37.39)(7.62,36.19) \qbezier(7.62,36.19)(5.22,34.98)(0,35) \linethickness{0.3mm} \qbezier(60,15)(54.78,15.02)(52.38,13.81) \qbezier(52.38,13.81)(49.97,12.61)(50,10) \qbezier(50,10)(49.97,7.39)(52.38,6.19) \qbezier(52.38,6.19)(54.78,4.98)(60,5) \linethickness{0.3mm} \put(0,75){\line(1,0){60}} \put(0,0){\line(0,1){75}} \put(60,0){\line(0,1){75}} \put(0,0){\line(1,0){60}} \linethickness{0.3mm} \put(15,70){\line(0,1){5}} \linethickness{0.3mm} \put(20,70){\line(1,0){25}} \linethickness{0.3mm} \put(20,70){\line(0,1){5}} \linethickness{0.3mm} \put(45,70){\line(0,1){5}} \put(30,72.5){\makebox(0,0)[cc]{$\cdots$}}
\put(55,46){\makebox(0,0)[cc]{$\vdots$}}
\put(5,61){\makebox(0,0)[cc]{$\vdots$}}
\put(-5,25.5){\makebox(0,0)[cc]{$2r\left\{\vbox to1.7cm{\vfill}\right.$}}
\put(25.75,2.5){\makebox(0,0)[cc]{*}}
\put(25.75,45){\makebox(0,0)[cc]{*}}
\put(64.75,10){\makebox(0,0)[cc]{$\Bigg\}2s$}}

\end{picture}}\begin{equation}
\delta^{-(r+s)/2}\ \ \qquad\raisebox{-2.5cm}[3cm]{\figurethree}\label{eq:inclusions}\end{equation}

Let\[
V_{\epsilon_{2}}^{\epsilon_{1}}(p)=\bigcup_{{a,b\geq0\atop \epsilon_{2}=(-1)^{a}\epsilon_{1}}}V_{a,b}^{\epsilon_{1}}(p)\]
and\[
V_{\epsilon_{2}}^{\epsilon_{1}}=\bigoplus_{p\geq0}V_{\epsilon_{2}}^{\epsilon_{1}}(p).\]
Finally, we let\[
V_{+}=V_{+}^{+}\oplus V_{+}^{-}.\]

One easily checks that the inclusions \eqref{eq:inclusions} are compatible
with the multiplication \eqref{eq:multiplication} and the trace \eqref{eq:trace},
thus proving the following:
\begin{lem}
\label{lem:defofA}(a) Equation \eqref{eq:multiplication} determines
an associative multiplication $\cdot$ on $V_{+}=V_{+}^{+}+V_{+}^{-}$
so that \[
V_{\epsilon}^{\epsilon'}(p)\cdot V_{\epsilon}^{\epsilon''}(p')\subset V_{\epsilon}^{\epsilon'}(p+p').\]
Furthermore, if $x\in V_{+}^{\epsilon}(0)$ and $y\in V_{+}^{\epsilon'}$
and $\epsilon\neq\epsilon'$ , then $x\cdot y=0$.\\
(b) Equation \eqref{eq:trace} defines a trace on $V_{+}$.\\
(c) The linear spaces $A_{+}^{\epsilon}\stackrel{\operatorname{def}}{=}V_{+}^{\epsilon}(0)$
form subalgebras of $V_{+}$. Moreover, $A_{+}=A_{+}^{+}+A_{+}^{-}$
is isomorphic to $A_{+}^{+}\oplus A_{+}^{-}$ as algebras.
\end{lem}
We could also define $V_{-}=V_{-}^{-}\oplus V_{-}^{+}$; this is also
an algebra with a trace, in the analogous way. 
\begin{lem}
\label{lem:propertiesofA}(a) The tangle\def\figurefour{\unitlength 0.75mm \begin{picture}(60,35)(0,35) \linethickness{0.3mm} \put(10,60){\line(1,0){40}} \put(10,45){\line(0,1){15}} \put(50,45){\line(0,1){15}} \put(10,45){\line(1,0){40}} \linethickness{0.3mm} \put(0,65){\line(1,0){60}} \linethickness{0.3mm} \put(50,57.5){\line(1,0){10}} \linethickness{0.3mm} \put(50,47.5){\line(1,0){10}} \linethickness{0.3mm} \put(0,57.5){\line(1,0){10}} \linethickness{0.3mm} \put(0,47.5){\line(1,0){10}} \put(55,54){\makebox(0,0)[cc]{$\vdots$}}
\put(5,54){\makebox(0,0)[cc]{$\vdots$}}
\linethickness{0.3mm} \put(0,70){\line(1,0){60}} \put(0,40){\line(0,1){30}} \put(60,40){\line(0,1){30}} \put(0,40){\line(1,0){60}} \put(10,42.5){\makebox(0,0)[cc]{$*$}}
\put(5,37.5){\makebox(0,0)[cc]{$*$}}
\put(5,37.5){\makebox(0,0)[cc]{}}
\end{picture} }\begin{equation}
\raisebox{-1.2cm}[1.8cm]\figurefour\label{eq:inclusion}\end{equation}
defines an injection $i$ from $A_{+}^{+}$ to $A_{+}^{-}$. \\
(b) The tangle \def\figurefive{\unitlength 0.75mm 
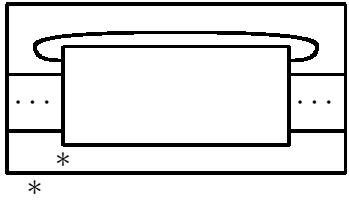
\begin{picture}(60,35)(0,35) \linethickness{0.3mm} \put(10,62.5){\line(1,0){40}} \put(10,45){\line(0,1){17.5}} \put(50,45){\line(0,1){17.5}} \put(10,45){\line(1,0){40}} \linethickness{0.3mm} \put(50,57.5){\line(1,0){10}} \linethickness{0.3mm} \put(50,47.5){\line(1,0){10}} \linethickness{0.3mm} \put(0,57.5){\line(1,0){10}} \linethickness{0.3mm} \put(0,47.5){\line(1,0){10}} \put(55,52.5){\makebox(0,0)[cc]{$\cdots$}}
\put(5,52.5){\makebox(0,0)[cc]{$\cdots$}}
\linethickness{0.3mm} \put(0,70){\line(1,0){60}} \put(0,40){\line(0,1){30}} \put(60,40){\line(0,1){30}} \put(0,40){\line(1,0){60}} \put(10,42.5){\makebox(0,0)[cc]{$*$}}
\put(5,37.5){\makebox(0,0)[cc]{$*$}}
\put(5,37.5){\makebox(0,0)[cc]{}}
\linethickness{0.3mm} \qbezier(50,60)(52.61,59.99)(53.81,60.59) \qbezier(53.81,60.59)(55.02,61.2)(55,62.5) \qbezier(55,62.5)(55.08,63.8)(49.06,64.41) \qbezier(49.06,64.41)(43.05,65.01)(30,65) \qbezier(30,65)(16.95,65.01)(10.94,64.41) \qbezier(10.94,64.41)(4.92,63.8)(5,62.5) \qbezier(5,62.5)(4.98,61.2)(6.19,60.59) \qbezier(6.19,60.59)(7.39,59.99)(10,60) \end{picture}}\begin{equation}
\raisebox{-1.2cm}[1.8cm]\figurefive\label{eq:E1}\end{equation}
defines a completely-positive map $E_{1}:A_{+}^{-}\to A_{+}^{+}$.\\
(c) Let $\eta:A_{+}\to A_{+}$ be given by\[
\eta(a\oplus b)=E_{1}(b)\oplus i(a),\qquad a\in A_{+}^{+},b\in A_{+}^{-}.\]
Then $Tr(x\eta(y))=Tr(\eta(x)y)$ for all $x,y\in A_{+}$.\\
(d) The inclusion of $V_{a,a}^{\epsilon}(0)$ into $V_{a+2,a+2}^{\epsilon}(0)$
determined by setting $r=s=1$ in \eqref{eq:inclusions} is the same
as the inclusion $\alpha$ described in Lemma~2.1 in \cite{shlyakht-popa:universal}.
In particular, the algebra $A_{+}^{+}$ (resp., $A_{+}^{-}$) is exactly
(the algebraic inductive limit of $\mathcal{P}_{k}$'s inside) the
type I von Neumann algebra $\mathcal{A}_{-1}^{-1}$ (resp., $\mathcal{A}_{0}^{-1}$)
defined in (2.4.7) in \cite{shlyakht-popa:universal}, and this identification
is trace-preserving.\\
(e) The trace $Tr$ is non-negative definite on $A_{+}$, and moreover
there exists a type I semi-finite von Neumann algebra $\mathfrak{A}_{+}$
with trace $Tr$ containing $A_{+}$ as a weakly dense subalgebra
in a trace-preserving way. The minimal projections of $\mathfrak{A}_{+}$
are contained in $A_{+}$$ $.\\
(f) The minimal central projections of $A_{+}$ are labeled by
the graph $\Gamma$. The trace of a minimal projection of $A_{+}$
lying in the central component associated to the vertex $v$ is the
value of the Perron-Frobenius eigenvector $\mu(v)$, normalized by
$\mu(*)=1$.\\
(g) If $\mathfrak{A}_{+}^{\epsilon}$ is the closure of $A_{+}^{\epsilon}$
in $\mathfrak{A}_{+}$, then $\mathfrak{A}_{+}=\mathfrak{A}_{+}^{+}\oplus\mathfrak{A}_{+}^{-}$.
Moreover, $\mathfrak{A}_{+}^{\epsilon}=P_{\epsilon}\mathfrak{A}_{+}$,
where $P_{\epsilon}$ is the central projection corresponding to all
even (if $\epsilon=+$) or odd (if $\epsilon=-$) vertices in the
principal graph $\Gamma$.\\
(h) The inclusion $\mathfrak{A}_{+}^{+}\subset\mathfrak{A_{+}^{-}}$
is given by the graph $\Gamma$.\end{lem}
\begin{proof}
Parts (a), (b) and (c), (d) are straightforward. Both (e) and (f)
follow from (d); indeed one takes $\mathfrak{A}_{+}=\mathcal{A}_{-1}^{-1}\oplus\mathcal{A}_{0}^{-1}$.
The remaining parts (f), (g) and (h) follow from \cite{shlyakht-popa:universal},
Lemma 2.6.
\end{proof}
One could instead work with $A_{-}$; in this case a similar lemma
holds, with the exception of replacing $\Gamma$ by the dual principal
graph.
\begin{thm}
\label{thm:compression}Let $e_{n}\in A_{+}^{(-1)^{n}}$ denote the
projection \def\figuresix{\unitlength 0.5mm \begin{picture}(63.75,75)(0,0) \linethickness{0.3mm} \put(10,70){\line(1,0){50}} \linethickness{0.3mm} \put(10,65){\line(1,0){50}} \linethickness{0.3mm} \put(10,55){\line(1,0){50}} \linethickness{0.3mm} \qbezier(10,45)(15.22,45.02)(17.62,43.81) \qbezier(17.62,43.81)(20.03,42.61)(20,40) \qbezier(20,40)(20.03,37.39)(17.62,36.19) \qbezier(17.62,36.19)(15.22,34.98)(10,35) \linethickness{0.3mm} \qbezier(10,30)(15.22,30.02)(17.62,28.81) \qbezier(17.62,28.81)(20.03,27.61)(20,25) \qbezier(20,25)(20.03,22.39)(17.62,21.19) \qbezier(17.62,21.19)(15.22,19.98)(10,20) \linethickness{0.3mm} \qbezier(10,15)(15.22,15.02)(17.62,13.81) \qbezier(17.62,13.81)(20.03,12.61)(20,10) \qbezier(20,10)(20.03,7.39)(17.62,6.19) \qbezier(17.62,6.19)(15.22,4.98)(10,5) \linethickness{0.3mm} \qbezier(60,45)(54.78,45.02)(52.38,43.81) \qbezier(52.38,43.81)(49.97,42.61)(50,40) \qbezier(50,40)(49.97,37.39)(52.38,36.19) \qbezier(52.38,36.19)(54.78,34.98)(60,35) \linethickness{0.3mm} \qbezier(60,30)(54.78,30.02)(52.38,28.81) \qbezier(52.38,28.81)(49.97,27.61)(50,25) \qbezier(50,25)(49.97,22.39)(52.38,21.19) \qbezier(52.38,21.19)(54.78,19.98)(60,20) \linethickness{0.3mm} \qbezier(60,15)(54.78,15.02)(52.38,13.81) \qbezier(52.38,13.81)(49.97,12.61)(50,10) \qbezier(50,10)(49.97,7.39)(52.38,6.19) \qbezier(52.38,6.19)(54.78,4.98)(60,5) \linethickness{0.3mm} \put(10,75){\line(1,0){50}} \put(10,0){\line(0,1){75}} \put(60,0){\line(0,1){75}} \put(10,0){\line(1,0){50}} \put(35,62){\makebox(0,0)[cc]{$\vdots$}}
\put(66,62){\makebox(0,0)[cc]{$\bigg\}n$}}
\end{picture}}\begin{equation}
\raisebox{-1.3cm}[2.6cm]\figuresix\label{eq:en}\end{equation}
Then there are a trace-preserving isomorphisms:\begin{eqnarray*}
(e_{n}V_{+}e_{n},\cdot,\delta^{-n}Tr) & \cong & \begin{cases}
(\mathcal{P},\wedge_{n},Tr_{n}), & n\textrm{ even},\\
(\mathcal{P}^{\operatorname{op}},\wedge_{n},Tr_{n}), & n\textrm{ odd};\end{cases}\\
(e_{n}V_{-}e_{n},\cdot,\delta^{-n}Tr) & \cong & \begin{cases}
(\mathcal{P}^{\operatorname{op}},\wedge_{n},Tr_{n}), & n\textrm{ even},\\
(\mathcal{P},\wedge_{n},Tr_{n}), & n\textrm{ odd};\end{cases}\end{eqnarray*}
where we write $\mathcal{P}^{\operatorname{op}}$ for the dual planar
algebra to $\mathcal{P}$ (i.e., one for which the shadings are reversed).\end{thm}
\begin{proof}
The isomorphism is given by identifying an element $x\in\mathcal{P}_{2n+p}^{\epsilon}$
with an element of $V_{n,n}^{\epsilon}(p)$, and then identifying
$V_{n,n}^{\epsilon}(p)$ with $e_{n}V_{(-1)^{n}\epsilon}^{\epsilon}e_{n}$.
\end{proof}

\section{Operator-valued semicircular systems.}

Let $A$ be a von Neumann algebra and let $\eta:A\to A$ be a completely-positive
map. Then \cite{shlyakht:semicirc,speicher:thesis} there is a unique
von Neumann algebra $M\supset A$, a conditional expectation $E:M\to A$
and an element $X=X^{*}\in M$ so that: (i) $M=W^{*}(A,X)$; (ii)
$E(a_{0}Xa_{1})=0$ for any $a_{j}\in A$ and $E$ satisfies the following
recursive property (here $a_{1},a_{2},\dots\in A$):\begin{equation}
E(a_{0}Xa_{1}\cdots Xa_{n})=\sum_{k=2}^{n}a_{0}\eta(E(a_{1}Xa_{2}\cdots Xa_{k-1}))\ E(a_{k}Xa_{k+1}\cdots a_{n-1}Xa_{n}).\label{eq:recursiveDefOfX}\end{equation}
For example,\[
E(a_{0}Xa_{1})=0,\qquad E(a_{0}Xa_{1}Xa_{2})=a_{0}\eta(a_{1})a_{2}.\]
The element $X$ is called the $A$-valued semicircular element with
variance $\eta$. Note that $\eta$ is determined by $\eta(a)=E(XaX)$.

The definition still makes sense if $A$ is a self-adjoint subalgebra
of a von Neumann algebra, provided that $\eta$ extends to a completely-positive
map on (some) the von Neumann algebra containing $A$. The $A$-valued
distribution of $X$ is completely described by the recursive formula
\eqref{eq:recursiveDefOfX}.

It is not hard to see that is equivalent to the following graphical
rule of computing $E(a_{0}Xa_{1}\cdots Xa_{n})$. First draw the product
as follows:\def\myfigureX{\unitlength 0.9 mm 
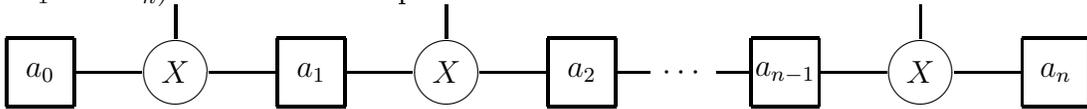
\begin{picture}(160,85)(0,70) \linethickness{0.3mm} \put(0,80){\line(1,0){10}} \put(0,70){\line(0,1){10}} \put(10,70){\line(0,1){10}} \put(0,70){\line(1,0){10}} \linethickness{0.3mm} \put(25,75){\circle{10}}
\linethickness{0.3mm} \put(65,75){\circle{10}}
\linethickness{0.3mm} \put(40,80){\line(1,0){10}} \put(40,70){\line(0,1){10}} \put(50,70){\line(0,1){10}} \put(40,70){\line(1,0){10}} \linethickness{0.3mm} \put(80,80){\line(1,0){10}} \put(80,70){\line(0,1){10}} \put(90,70){\line(0,1){10}} \put(80,70){\line(1,0){10}} \linethickness{0.3mm} \put(110,80){\line(1,0){10}} \put(110,70){\line(0,1){10}} \put(120,70){\line(0,1){10}} \put(110,70){\line(1,0){10}} \linethickness{0.3mm} \put(135,75){\circle{10}}
\linethickness{0.3mm} \put(150,80){\line(1,0){10}} \put(150,70){\line(0,1){10}} \put(160,70){\line(0,1){10}} \put(150,70){\line(1,0){10}} \linethickness{0.3mm} \put(10,75){\line(1,0){10}} \linethickness{0.3mm} \put(30,75){\line(1,0){10}} \linethickness{0.3mm} \put(50,75){\line(1,0){10}} \linethickness{0.3mm} \put(70,75){\line(1,0){10}} \linethickness{0.3mm} \put(90,75){\line(1,0){5}} \linethickness{0.3mm} \put(105,75){\line(1,0){5}} \linethickness{0.3mm} \put(120,75){\line(1,0){10}} \linethickness{0.3mm} \put(140,75){\line(1,0){10}} \linethickness{0.3mm} \put(25,80){\line(0,1){5}} \linethickness{0.3mm} \put(65,80){\line(0,1){5}} \linethickness{0.3mm} \put(135,80){\line(0,1){5}} \put(5,75){\makebox(0,0)[cc]{$a_0$}}
\put(45,75){\makebox(0,0)[cc]{$a_1$}}
\put(85,75){\makebox(0,0)[cc]{$a_2$}}
\put(115,75){\makebox(0,0)[cc]{$a_{n-1}$}}
\put(155,75){\makebox(0,0)[cc]{$a_{n}$}}
\put(25,75){\makebox(0,0)[cc]{$X$}}
\put(65,75){\makebox(0,0)[cc]{$X$}}
\put(135,75){\makebox(0,0)[cc]{$X$}}
\put(100,75){\makebox(0,0)[cc]{$\cdots$}}
\put(130,30){\makebox(0,0)[cc]{}}
\put(130,35){\makebox(0,0)[cc]{}}
\end{picture}}\[
\raisebox{0pt}[1.1cm]\myfigureX\]
Next, consider the following drawing, where $\sum TL$ stands for
the sum over all $T$:\def\myfigureXX{\unitlength 0.9mm 
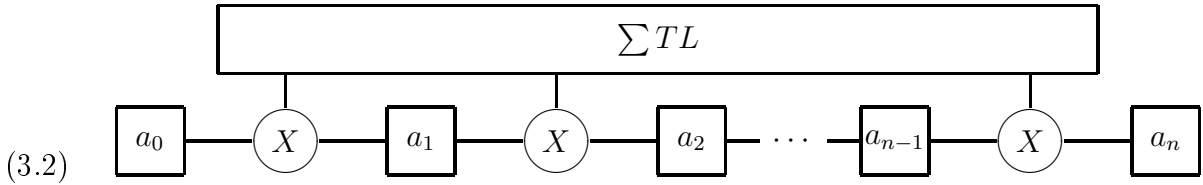
\begin{picture}(160,95)(0,70) \linethickness{0.3mm} \put(0,80){\line(1,0){10}} \put(0,70){\line(0,1){10}} \put(10,70){\line(0,1){10}} \put(0,70){\line(1,0){10}} \linethickness{0.3mm} \put(25,75){\circle{10}}
\linethickness{0.3mm} \put(65,75){\circle{10}}
\linethickness{0.3mm} \put(40,80){\line(1,0){10}} \put(40,70){\line(0,1){10}} \put(50,70){\line(0,1){10}} \put(40,70){\line(1,0){10}} \linethickness{0.3mm} \put(80,80){\line(1,0){10}} \put(80,70){\line(0,1){10}} \put(90,70){\line(0,1){10}} \put(80,70){\line(1,0){10}} \linethickness{0.3mm} \put(110,80){\line(1,0){10}} \put(110,70){\line(0,1){10}} \put(120,70){\line(0,1){10}} \put(110,70){\line(1,0){10}} \linethickness{0.3mm} \put(135,75){\circle{10}}
\linethickness{0.3mm} \put(150,80){\line(1,0){10}} \put(150,70){\line(0,1){10}} \put(160,70){\line(0,1){10}} \put(150,70){\line(1,0){10}} \linethickness{0.3mm} \put(10,75){\line(1,0){10}} \linethickness{0.3mm} \put(30,75){\line(1,0){10}} \linethickness{0.3mm} \put(50,75){\line(1,0){10}} \linethickness{0.3mm} \put(70,75){\line(1,0){10}} \linethickness{0.3mm} \put(90,75){\line(1,0){5}} \linethickness{0.3mm} \put(105,75){\line(1,0){5}} \linethickness{0.3mm} \put(120,75){\line(1,0){10}} \linethickness{0.3mm} \put(140,75){\line(1,0){10}} \linethickness{0.3mm} \put(25,80){\line(0,1){5}} \linethickness{0.3mm} \put(65,80){\line(0,1){5}} \linethickness{0.3mm} \put(135,80){\line(0,1){5}} \put(5,75){\makebox(0,0)[cc]{$a_0$}}
\put(45,75){\makebox(0,0)[cc]{$a_1$}}
\put(85,75){\makebox(0,0)[cc]{$a_2$}}
\put(115,75){\makebox(0,0)[cc]{$a_{n-1}$}}
\put(155,75){\makebox(0,0)[cc]{$a_{n}$}}
\put(25,75){\makebox(0,0)[cc]{$X$}}
\put(65,75){\makebox(0,0)[cc]{$X$}}
\put(135,75){\makebox(0,0)[cc]{$X$}}
\put(100,75){\makebox(0,0)[cc]{$\cdots$}}
\put(130,30){\makebox(0,0)[cc]{}}
\put(130,35){\makebox(0,0)[cc]{}}
\linethickness{0.3mm} \put(15,95){\line(1,0){130}} \put(15,85){\line(0,1){10}} \put(145,85){\line(0,1){10}} \put(15,85){\line(1,0){130}} \put(80,90){\makebox(0,0)[cc]{$\sum TL$}}
\put(145,35){\makebox(0,0)[cc]{}}
\end{picture}}\begin{equation}
\raisebox{0pt}[2.4cm]\myfigureXX\label{eq:SumoverTLandOpVal}\end{equation}
Finally, obtain the value of $E(a_{0}Xa_{1}X\cdots a_{n-1}Xa_{n})$
by recursively performing the following replacements in \eqref{eq:SumoverTLandOpVal}:\def\mypictureXXX{\unitlength 0.7 mm 
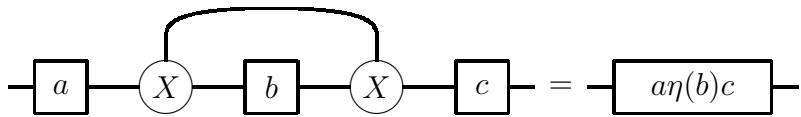
\begin{picture}(155,90)(0,80) \linethickness{0.3mm} \put(10,80){\line(1,0){10}} \put(10,70){\line(0,1){10}} \put(20,70){\line(0,1){10}} \put(10,70){\line(1,0){10}} \linethickness{0.3mm} \put(35,75){\circle{10}}
\linethickness{0.3mm} \put(50,80){\line(1,0){10}} \put(50,70){\line(0,1){10}} \put(60,70){\line(0,1){10}} \put(50,70){\line(1,0){10}} \linethickness{0.3mm} \put(75,75){\circle{10}}
\linethickness{0.3mm} \put(90,80){\line(1,0){10}} \put(90,70){\line(0,1){10}} \put(100,70){\line(0,1){10}} \put(90,70){\line(1,0){10}} \put(15,75){\makebox(0,0)[cc]{$a$}}
\put(55,75){\makebox(0,0)[cc]{$b$}}
\put(95,75){\makebox(0,0)[cc]{$c$}}
\linethickness{0.3mm} \put(120,80){\line(1,0){30}} \put(120,70){\line(0,1){10}} \put(150,70){\line(0,1){10}} \put(120,70){\line(1,0){30}} \put(135,75){\makebox(0,0)[cc]{$a\eta(b)c$}}
\put(110,75){\makebox(0,0)[cc]{$=$}}
\put(35,75){\makebox(0,0)[cc]{$X$}}
\put(75,75){\makebox(0,0)[cc]{$X$}}
\linethickness{0.3mm} \put(35,80){\line(0,1){5}} \linethickness{0.3mm} \put(75,80){\line(0,1){5}} \linethickness{0.3mm} \qbezier(35,85)(34.94,87.61)(39.75,88.81) \qbezier(39.75,88.81)(44.56,90.02)(55,90) \qbezier(55,90)(65.44,90.02)(70.25,88.81) \qbezier(70.25,88.81)(75.06,87.61)(75,85) \linethickness{0.3mm} \put(5,75){\line(1,0){5}} \linethickness{0.3mm} \put(20,75){\line(1,0){10}} \linethickness{0.3mm} \put(40,75){\line(1,0){10}} \linethickness{0.3mm} \put(60,75){\line(1,0){10}} \linethickness{0.3mm} \put(80,75){\line(1,0){10}} \linethickness{0.3mm} \put(100,75){\line(1,0){5}} \linethickness{0.3mm} \put(115,75){\line(1,0){5}} \linethickness{0.3mm} \put(150,75){\line(1,0){5}} \end{picture} }\[
\raisebox{0pt}[1cm]\mypictureXXX\]
In the case that $A$ is semi-finite with trace $Tr$ and $\eta$
satisfies\[
Tr(x\eta(y))=Tr(\eta(x)y),\qquad\forall x,y\in\mathcal{L}^{1}(Tr)\]
the algebra $M$ is also semi-finite with trace $Tr\circ E$.

More generally, a family $\{X_{i}:i\in I\}$ is called semicircular
over $A$ if one has the recursive relation\[
E(a_{0}X_{i_{1}}a_{1}\cdots X_{i_{n}}a_{n})=\sum_{k=2}^{n}a_{0}\eta_{i_{1}i_{k}}(E(a_{1}X_{i_{2}}a_{2}\cdots Xa_{k-1}))\ E(a_{k}X_{i_{k+1}}a_{k+1}\cdots a_{n-1}X_{i_{n}}a_{n}).\]
The joint variance $\eta_{ij}$ can be viewed as a matrix-valued map
$\eta=(\eta_{ij})_{ij}:A\to A\otimes B(\ell^{2}(I))$ (where $B(\ell^{2}(I))$
stands for bounded operators on $\ell^{2}(I)$), and the positivity
requirement is that this map be completely-positive.

We need the following Lemma, which can be found in \cite{shlyakht:semicirc,shlyakht:amalg}:
\begin{lem}
\label{lem:OpValSemicirc}(a) Let $X$ be an $A$-valued semicircular
element, let $a_{i},b_{i}\in A$. Then the elements $\{a_{i}^{*}Xb_{j}+b_{j}^{*}Xa_{i}\}_{i\leq j}$
form an $A$-valued semicircular family. (b) Let $X_{i}:i\in I$ be
an $A$-valued semicircular family. Then $X_{i}$ are free with amalgamation
over $A$ iff $E(X_{i}aX_{j})=0$ for $i\neq j$ and all $a\in A$.
(c) If $A_{1}\subset A$ is a subalgebra so that $\eta(A_{1})\subset A_{1}$,
then $X$ is also $A_{1}$-semicircular. (d) If $\eta(a)=\tau(a)1$
for all $a\in A$ and a state $\tau:A\to\mathbb{C}$, then $X$ is
free from $A$ in $W^{*}(A,X,\tau\circ E)$. (e) $\Vert X\Vert\leq2\Vert\eta(1)\Vert$.
\end{lem}
We now return to the algebra $A_{+}=A_{+}^{+}\oplus A_{+}^{-}$ that
we defined in Lemma \ref{lem:defofA}. Let $\mathfrak{A}_{+}=\mathfrak{A}_{+}^{+}\oplus\mathfrak{A}_{-}^{+}$
be as in Lemma \ref{lem:propertiesofA}(e). Let $\eta:\mathfrak{A}_{+}\to\mathfrak{A}_{+}$
be given by\[
\eta(a\oplus b)=E_{1}(b)\oplus i(a)\]
 as in Lemma \eqref{lem:defofA}. Then $\eta$ is a completely-positive
map and satisfies $Tr(x\eta(y))=Tr(\eta(x)y)$ for all $x,y\in\mathcal{L}^{1}(Tr)$.
Let $X$ be an $\mathfrak{A}_{+}$-valued semicircular element, and
let $\mathfrak{M}_{+}=W^{*}(\mathfrak{A}_{+},X)$, $E:\mathfrak{M}_{+}\to\mathfrak{A}_{+}$
the canonical conditional expectation, and $Tr=Tr\circ E$ a semi-finite
trace on $\mathfrak{M}_{+}$. 
\begin{lem}
Let $X_{n}=e_{2n}Xe_{2n+1}+e_{2n+1}Xe_{2n}$. Let also $f_{n}=e_{2n}\oplus e_{2n+1}\in A_{+}$.
Then $f_{n}\uparrow1$ weakly. Moreover, $X_{n}=f_{n}Xf_{n}$, and
$X_{n}$ is an operator-valued semicircular element over $\mathfrak{A}_{+}$
with variance $\eta_{n}(x)=f_{n}\eta(f_{n}xf_{n})f_{n}$. In particular,
since $f_{n}\mathfrak{A}_{+}f_{n}=f_{n}A_{+}f_{n}$ and $\eta_{n}(f_{n}A_{+}f_{n})\subset f_{n}A_{+}f_{n}$,
$X_{n}$ is $f_{n}Af_{n}$-semicircular with variance $\eta|_{f_{n}Af_{n}}$. 
\end{lem}
Consider now the element $c_{n}\in V_{+}$ given by the diagram\def\figureseven{\unitlength 0.7mm \begin{picture}(60,40)(0,40) \linethickness{0.3mm} \put(10,60){\line(1,0){50}} \linethickness{0.3mm} \put(10,55){\line(1,0){50}} \linethickness{0.3mm} \put(10,45){\line(1,0){50}} \linethickness{0.3mm} \put(10,75){\line(1,0){50}} \put(10,40){\line(0,1){35}} \put(60,40){\line(0,1){35}} \put(10,40){\line(1,0){50}} \put(35,51){\makebox(0,0)[cc]{$\vdots$}}
\put(70.5,55){\makebox(0,0)[cc]{$\left.\vbox to1cm{\vfill}\right\}2n+1$}}
\linethickness{0.3mm} 
\qbezier(60,65)(44.34,64.99)(37.12,65) 
\qbezier(37.12,65)(29.91,66.2)(30,71) 
\qbezier(30,71)(30,72.4)(30,75) 
\end{picture}}\begin{equation}
c_{n}=\!\!\!\!\!\raisebox{-0.8cm}[1.6cm]\figureseven\label{eq:defofCornern}\end{equation}
(the top-left corner is unshaded).
\begin{lem}
(a) $e_{2n}c_{n}=c_{n}e_{2n+1}$;\\
(b) $f_{m}c_{n}f_{m}=c_{m}$ if $n\geq m$.\\
(c) If $x_{n}=c_{n}+c_{n}^{*}$, then $c_{n}=e_{2n}x_{n}e_{2n+1}$.\\
(d) For any $N_{0}$, $V_{+}=\operatorname{Alg}(A_{+,}\{x_{n}\}_{n\geq N_{0}})$.\\
(e) Let $E:V_{+}\to A_{+}$ be the conditional expectation given
by the tangle\def\figureeight{\unitlength 0.7mm \begin{picture}(50,35)(0,39) \linethickness{0.3mm} \put(10,60){\line(1,0){30}} \put(10,40){\line(0,1){20}} \put(40,40){\line(0,1){20}} \put(10,40){\line(1,0){30}} \linethickness{0.3mm} \put(40,55){\line(1,0){10}} \linethickness{0.3mm} \put(0,55){\line(1,0){10}} \linethickness{0.3mm} \put(40,45){\line(1,0){10}} \linethickness{0.3mm} \put(0,45){\line(1,0){10}} 
\put(45,51){\makebox(0,0)[cc]{$\vdots$}}
\put(5,51){\makebox(0,0)[cc]{$\vdots$}}
\linethickness{0.3mm} \put(10,72.5){\line(1,0){30}} \put(10,65){\line(0,1){7.5}} \put(40,65){\line(0,1){7.5}} \put(10,65){\line(1,0){30}} \linethickness{0.3mm} \put(15,60){\line(0,1){5}} \linethickness{0.3mm} \put(35,60){\line(0,1){5}} \put(25,62.5){\makebox(0,0)[cc]{$\cdots$}}
\put(25,68.75){\makebox(0,0)[cc]{$\sum TL$}}
\linethickness{0.3mm} \put(0,75){\line(1,0){50}} \put(0,37.5){\line(0,1){37.5}} \put(50,37.5){\line(0,1){37.5}} \put(0,37.5){\line(1,0){50}} \end{picture}}\begin{equation}
\raisebox{-0.8cm}[1.6cm]\figureeight\label{eq:defofE}\end{equation}
Then $x_{n}$ is $A_{+}$-semicircular with variance $\eta_{n}(x)=f_{n}\eta(f_{n}xf_{n})f_{n}$.
In particular, $E(x_{n}ax_{n}f_{n})=f_{n}\eta(f_{n}xf_{n})f_{n}$
for all $a\in A_{+}$. \end{lem}
\begin{proof}
(a), (b), (c) are straightforward verifications. 

To prove (d), we clearly have the inclusion $\operatorname{Alg}(A_{+,}\{x_{n}\}_{n\geq N_{0}})\subset V_{+}$.
Note that $c_{n}$ and $c_{n}^{*}$ belong to $\operatorname{Alg}(A_{+},x_{n})$.
Next, one can obtain any element of $V_{+}$ by applying the following
tangle to various elements of $A_{+}$:\def\figurenine{\unitlength 1mm 
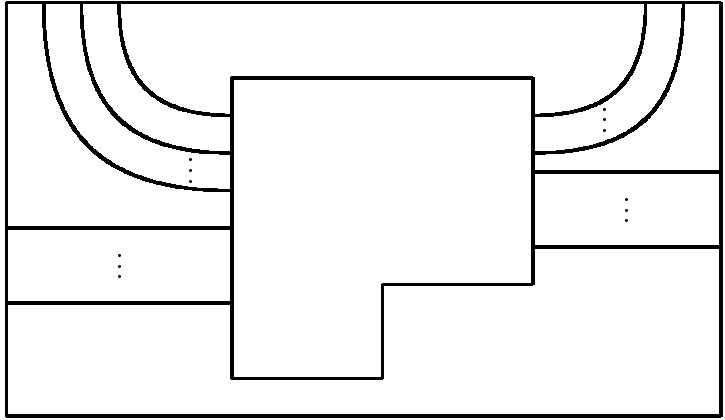
\begin{picture}(80,50)(0,20) \linethickness{0.3mm} \put(30,22.5){\line(0,1){40}} \linethickness{0.3mm} \put(30,22.5){\line(1,0){20}} \linethickness{0.3mm} \put(50,22.5){\line(0,1){12.5}} \linethickness{0.3mm} \put(50,35){\line(1,0){20}} \linethickness{0.3mm} \put(70,35){\line(0,1){27.5}} \linethickness{0.3mm} \put(30,62.5){\line(1,0){40}} \linethickness{0.3mm} \qbezier(70,57.5)(77.83,57.45)(81.44,61.06) \qbezier(81.44,61.06)(85.05,64.67)(85,72.5) \linethickness{0.3mm} \qbezier(70,52.5)(80.44,52.44)(85.25,57.25) \qbezier(85.25,57.25)(90.06,62.06)(90,72.5) \linethickness{0.3mm} \qbezier(30,57.5)(22.17,57.45)(18.56,61.06) \qbezier(18.56,61.06)(14.95,64.67)(15,72.5) \linethickness{0.3mm} \qbezier(30,52.5)(19.56,52.44)(14.75,57.25) \qbezier(14.75,57.25)(9.94,62.06)(10,72.5) \linethickness{0.3mm} \put(70,50){\line(1,0){25}} \linethickness{0.3mm} \put(70,40){\line(1,0){25}} \linethickness{0.3mm} \qbezier(30,47.5)(16.95,47.42)(10.94,53.44) \qbezier(10.94,53.44)(4.92,59.45)(5,72.5) \linethickness{0.3mm} \put(0,42.5){\line(1,0){30}} \linethickness{0.3mm} \put(0,32.5){\line(1,0){30}} \linethickness{0.3mm} \put(0,72.5){\line(1,0){95}} \put(0,17.5){\line(0,1){55}} \put(95,17.5){\line(0,1){55}} \put(0,17.5){\line(1,0){95}} \put(82.5,46){\makebox(0,0)[cc]{$\vdots$}}
\put(15,38.5){\makebox(0,0)[cc]{$\vdots$}}
\put(77.5,57.5){\makebox(0,0)[cc]{}}
\put(79.5,58){\makebox(0,0)[cc]{$\vdots$}}
\put(24.5,51.25){\makebox(0,0)[cc]{$\vdots$}}
\end{picture}}\[
\raisebox{0pt}[5.1cm]\figurenine\]
This tangle, however, can be obtained by composing the multiplication
tangle with diagrams of the form \def\figureten{\unitlength 0.7mm \begin{picture}(40,40)(0,40) \linethickness{0.3mm} \qbezier(0,60)(5.22,59.97)(7.62,62.37) \qbezier(7.62,62.37)(10.03,64.78)(10,70) \linethickness{0.3mm} \qbezier(0,55)(7.83,54.95)(11.44,58.56) \qbezier(11.44,58.56)(15.05,62.17)(15,70) \linethickness{0.3mm} \qbezier(0,50)(10.44,49.94)(15.25,54.75) \qbezier(15.25,54.75)(20.06,59.56)(20,70) \linethickness{0.3mm} \put(0,45){\line(1,0){25}} \linethickness{0.3mm} \put(0,40){\line(1,0){25}} \put(25,40){\line(0,1){30}} \put(0,40){\line(0,1){30}} \put(0,70){\line(1,0){25}} \end{picture}}

\begin{equation}
\raisebox{-0.5cm}[1cm]\figureten\label{eq:Manyloops}\end{equation}
and their adjoints. But the following picture shows that \eqref{eq:Manyloops}
(up to the inductive limit defined by \eqref{eq:inclusions}) belongs
to $\operatorname{Alg}(A_{+,}\{x_{n}\}_{n\geq N_{0}})$:\def\figureeleven{\unitlength 0.7mm \begin{picture}(165,45)(0,20) \linethickness{0.3mm} \qbezier(-5,60)(0.22,59.97)(2.62,62.37) \qbezier(2.62,62.37)(5.03,64.78)(5,70) \linethickness{0.3mm} \qbezier(-5,55)(2.83,54.95)(6.44,58.56) \qbezier(6.44,58.56)(10.05,62.17)(10,70) \linethickness{0.3mm} \qbezier(125,70)(125.02,64.78)(123.81,62.37) \qbezier(123.81,62.37)(122.61,59.97)(120,60) \qbezier(120,60)(117.41,60)(115,60) \qbezier(115,60)(112.59,60)(110,60) \qbezier(110,60)(107.39,60.02)(106.19,58.81) \qbezier(106.19,58.81)(104.98,57.61)(105,55) \qbezier(105,55)(104.98,52.39)(106.19,51.19) \qbezier(106.19,51.19)(107.39,49.98)(110,50) \qbezier(110,50)(112.59,50)(115,50) \qbezier(115,50)(117.41,50)(120,50) \qbezier(120,50)(122.61,50.02)(123.81,48.81) \qbezier(123.81,48.81)(125.02,47.61)(125,45) \qbezier(125,45)(125.03,42.39)(122.62,41.19) \qbezier(122.62,41.19)(120.22,39.98)(115,40) \linethickness{0.3mm} \qbezier(75,70)(74.98,64.78)(76.19,62.37) \qbezier(76.19,62.37)(77.39,59.97)(80,60) \qbezier(80,60)(82.59,60)(85,60) \qbezier(85,60)(87.41,60)(90,60) \qbezier(90,60)(92.61,60.02)(93.81,58.81) \qbezier(93.81,58.81)(95.02,57.61)(95,55) \qbezier(95,55)(95.02,52.39)(93.81,51.19) \qbezier(93.81,51.19)(92.61,49.98)(90,50) \qbezier(90,50)(87.39,50)(86.19,50) \qbezier(86.19,50)(84.98,50)(85,50) \linethickness{0.3mm} \qbezier(65,70)(65.03,64.78)(62.62,62.37) \qbezier(62.62,62.37)(60.22,59.97)(55,60) \linethickness{0.3mm} \qbezier(145,50)(139.78,50.02)(137.38,48.81) \qbezier(137.38,48.81)(134.97,47.61)(135,45) \qbezier(135,45)(134.97,42.39)(137.38,41.19) \qbezier(137.38,41.19)(139.78,39.98)(145,40) \linethickness{0.3mm} \put(90,65){\line(1,0){20}} \put(90,45){\line(0,1){20}} \put(110,45){\line(0,1){20}} \put(90,45){\line(1,0){20}} \linethickness{0.3mm} \put(120,55){\line(1,0){20}} \put(120,35){\line(0,1){20}} \put(140,35){\line(0,1){20}} \put(120,35){\line(1,0){20}} \linethickness{0.3mm} \qbezier(-5,50)(5.44,49.94)(10.25,54.75) \qbezier(10.25,54.75)(15.06,59.56)(15,70) \linethickness{0.3mm} \put(55,50){\line(1,0){30}} \linethickness{0.3mm} \put(55,40){\line(1,0){60}} \linethickness{0.3mm} \qbezier(30,40)(24.78,40.02)(22.38,38.81) \qbezier(22.38,38.81)(19.97,37.61)(20,35) \qbezier(20,35)(19.97,32.39)(22.38,31.19) \qbezier(22.38,31.19)(24.78,29.98)(30,30) \put(40,45){\makebox(0,0)[cc]{$=$}}
\linethickness{0.3mm} \put(55,70){\line(1,0){110}} \put(55,20){\line(0,1){50}} \put(165,20){\line(0,1){50}} \put(55,20){\line(1,0){110}} \linethickness{0.3mm} \put(-5,45){\line(1,0){35}} \linethickness{0.3mm} \put(-5,70){\line(1,0){35}} \put(-5,25){\line(0,1){45}} \put(30,25){\line(0,1){45}} \put(-5,25){\line(1,0){35}} \linethickness{0.3mm} \put(55,30){\line(1,0){90}}
\put(55,30){\line(1,0){90}} \linethickness{0.3mm} \qbezier(145,40)(147.61,40.02)(148.81,38.81) \qbezier(148.81,38.81)(150.02,37.61)(150,35) \qbezier(150,35)(150.02,32.39)(148.81,31.19) \qbezier(148.81,31.19)(147.61,29.98)(145,30) \linethickness{0.3mm} \qbezier(160,40)(157.39,40.02)(156.19,38.81) \qbezier(156.19,38.81)(154.98,37.61)(155,35) \qbezier(155,35)(154.98,32.39)(156.19,31.19) \qbezier(156.19,31.19)(157.39,29.98)(160,30) \linethickness{0.3mm} \put(160,40){\line(1,0){5}} \linethickness{0.3mm} \put(160,30){\line(1,0){5}} \linethickness{0.3mm} \put(145,50){\line(1,0){20}} \linethickness{0.3mm} \put(145,45){\line(1,0){15}} \put(145,25){\line(0,1){20}} \put(160,25){\line(0,1){20}} \put(145,25){\line(1,0){15}} \end{picture} }\begin{equation}
\raisebox{-1.5cm}[2.3cm]\figureeleven.\label{eq:canGetManyLoops}\end{equation}
This shows that $\operatorname{Alg}(A_{+,}\{x_{n}\}_{n\geq N_{0}})=V_{+}$. 

Finally, part (e) follows from the definition of $E$ and \eqref{eq:SumoverTLandOpVal}.
\end{proof}
We have thus proved:
\begin{thm}
\label{thm:MPlusGenByAandX}The map taking $A_{+}$ into $\mathfrak{A}_{+}$
and each $x_{n}$ to $X_{n}$ extends to a trace-preserving isomorphism
between $(V_{+},\cdot,Tr)$ and a dense subalgebra of $\mathfrak{M}_{+}$.
Thus $\mathfrak{M}_{+}\cong W^{*}(\mathfrak{A}_{+},X)$, where $X$
is $\mathfrak{A}_{+}$-semicircular of variance $\eta$.\end{thm}
\begin{cor}
The trace $Tr$ on $V_{+}$ is non-negative definite, and elements
of $V_{+}$ give rise to bounded operators in the GNS representation
on $L^{2}(V_{+},Tr)$. \end{cor}
\begin{proof}
This follows from the fact that $L^{2}(V_{+},Tr)=L^{2}(\mathfrak{M}_{+},Tr)$,
and that every element of $V_{+}$ acts on $L^{2}(V_{+},Tr)$ in the
same way as some finite-degree non-commutative polynomial in $X$
and elements of $A_{+}$. Moreover, $\Vert X\Vert\leq2\Vert\eta(1)\Vert<\infty$.\end{proof}
\begin{cor}
\label{cor:Mjsascompressions}$M_{j}(\mathcal{P})\cong e_{j}\mathfrak{M}_{+}e_{j}$
if $j$ is even and $M_{j}(\mathcal{P}^{\operatorname{op}})\cong e_{j}\mathfrak{M}_{+}e_{j}$
if $j$ is odd. Here $\mathcal{P}^{\operatorname{op}}$ denotes the
dual planar algebra and we write $M_{j}(\mathcal{P})$ for the tower
of factors assocaited to a planar algebra $\mathcal{P}$ in \cite{guionnet-jones-shlyakhtenko1}.
\end{cor}
Let $P_{\epsilon}$ be the unit of $\mathfrak{A}_{+}^{\epsilon}$,
$\epsilon=\pm$. 
\begin{lem}
\label{lem:factoriality}(a) Let $c=P_{+}XP_{-}$ and let $c=vb$
be the polar decomposition of $c$. Then $vv^{*}=P_{+}$ and $v^{*}v\leq P_{-}$.
(b) $P_{\pm}$ both have central support 1 in $\mathfrak{M}_{+}$.
(c) If the index satisfies $\delta^{2}>1$, $\mathfrak{M}_{+}$ is
a type II$_{\infty}$ factor.\end{lem}
\begin{proof}
Let $b_{n}^{2}=c_{n}^{*}c_{n}$. Then $b_{n}\in e_{2n}\mathfrak{M}_{+}e_{2n}$
and its law there is the same as the law of $\cup$ in $M_{0}\subset M_{n}$.
It follows that if we set $v_{n}=c_{n}b_{n}^{-1/2}$, then $v_{n}v_{n}^{*}=e_{2n}$.
Since $e_{2n}\uparrow P_{+}$, $vv^{*}=P_{+}$. On the other hand,
since $e_{2n+1}c_{n}c_{n}^{*}e_{2n+1}=c_{n}c_{n}^{*}$, it follows
that $v_{n}^{*}v_{n}\leq e_{2n+1}$. Thus $vv^{*}\leq P_{-}.$ This
proves (a). To prove (b), consider a central projection $q$ so that
$q\geq P_{-}$. But then $q\geq v^{*}v$ and so $q\geq vv^{*}$ (since
$q$ is central). Thus $q\geq P_{+}$ also, and so $q\geq P_{+}+P_{-}=1$.
Finally, to prove (c), assume that $q\in\mathfrak{M}_{+}$ is a central
projection. Then since $[q,P_{\pm}]=0$, we find that $q=P_{+}qP_{+}+P_{-}qP_{-}=q_{+}+q_{-}$,
where $q_{\epsilon}\in P_{\epsilon}\mathfrak{M}_{+}P_{\epsilon}$.
Again, let $f_{n}=e_{2n}+e_{2n+1}$, $f_{n}\uparrow1$. Then $f_{n}q_{\epsilon}f_{n}$
is central in $P_{\epsilon}\mathfrak{M}P_{\epsilon}$, which is isomorphic
to either $M_{2n+1}(\mathcal{P}^{\operatorname{op}})$ or $M_{2n}(\mathcal{P})$
depending on the parity of $\epsilon$. But these are factors by \cite{guionnet-jones-shlyakhtenko1},
so $f_{n}q_{\epsilon}f_{n}$ must be multiples of identity. Since
$f_{n}\uparrow1$, it must be that $q=0$, $q=P_{+}$, $q=P_{-}$
or $q=1$. But because of (b) only $q=0$ and $q=1$ are possible. \end{proof}
\begin{cor}
\label{cor:stableisom}If the index satisfies $\delta^{2}>1$, we
have the following isomorphisms\[
M_{2j+1}(\mathcal{P}^{\operatorname{op}})\cong(M_{0})^{\delta^{2j+1}},\qquad M_{2j}(\mathcal{P})\cong(M_{0})^{\delta^{2j}}.\]
Here $\mathcal{P}^{\operatorname{op}}$ is the planar algebra dual
to $\mathcal{P}$ and we write $M_{j}(\mathcal{P})$ for the tower
of factors assocaited to a planar algebra $\mathcal{P}$ in \cite{guionnet-jones-shlyakhtenko1}.
\end{cor}
This is of course because by Corollary \ref{cor:Mjsascompressions},
$M_{1}(\mathcal{P}^{\operatorname{op}})\cong M_{0}(\mathcal{P})^{\delta}$
and in general $M_{k+2}\cong M_{k}^{\delta^{2}}$.

\section{The algebra $\mathfrak{M}_{+}$.}

Let $\Gamma$ be the principal graph of $\mathcal{P}$, and denote
by $\Gamma^{+}$ the even and by $\Gamma^{-}$ the odd vertices of
$\Gamma$. For each $v\in\Gamma^{+}$ choose a minimal projection
$q_{v}\in\mathfrak{A}_{+}^{+}$ in the central summand corresponding
to $v$ (in the case of $v=*$ we choose $q_{*}=e_{0}$). Let $\mu$
be the Perron-Frobenius eigenvector for $\Gamma$, normalized by $\mu(*)=1$.
Then $Q=\bigoplus_{v\in\Gamma^{+}}q_{v}\oplus i(q_{v})$ is a projection
in $\mathfrak{A}_{+}$. Moreover, there exists a family of partial
isometries $m_{i}\in\mathfrak{A}_{+}$ having the form $m_{i}=w_{i}\oplus i(w_{i})$
and satisfying: $m_{i}^{*}m_{i}=Q$, $\sum m_{i}m_{i}^{*}=1$. 

As in Theorem \ref{thm:MPlusGenByAandX}, let $X$ be an $\mathfrak{A}_{+}$-semicircular
element of variance $\eta$.
\begin{lem}
\label{lem:XcommutesWith}For any $a\in\mathfrak{A}_{+}^{+}$, $aX=Xi(a)$.\end{lem}
\begin{proof}
This follows from $aX_{n}=ac_{n}$ and the following diagram:\def\figurethirteen{\unitlength 0.7mm 
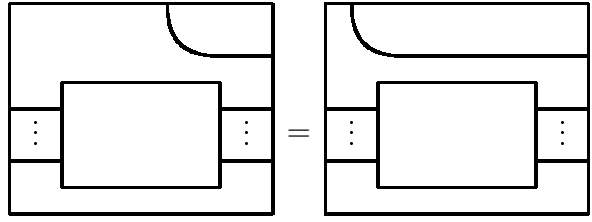
\begin{picture}(110,30)(0,40) \linethickness{0.3mm} \put(10,60){\line(1,0){30}} \put(10,40){\line(0,1){20}} \put(40,40){\line(0,1){20}} \put(10,40){\line(1,0){30}} \linethickness{0.3mm} \qbezier(30,75)(29.97,69.78)(32.38,67.38) \qbezier(32.38,67.38)(34.78,64.97)(40,65) \linethickness{0.3mm} \put(40,65){\line(1,0){10}} \linethickness{0.3mm} \put(40,55){\line(1,0){10}} \linethickness{0.3mm} \put(40,45){\line(1,0){10}} \linethickness{0.3mm} \put(0,55){\line(1,0){10}} \linethickness{0.3mm} \put(0,45){\line(1,0){10}} \linethickness{0.3mm} \put(0,75){\line(1,0){50}} \put(0,35){\line(0,1){40}} \put(50,35){\line(0,1){40}} \put(0,35){\line(1,0){50}} \linethickness{0.3mm} \put(70,60){\line(1,0){30}} \put(70,40){\line(0,1){20}} \put(100,40){\line(0,1){20}} \put(70,40){\line(1,0){30}} \linethickness{0.3mm} \qbezier(65,75)(64.97,69.78)(67.38,67.38) \qbezier(67.38,67.38)(69.78,64.97)(75,65) \linethickness{0.3mm} \put(75,65){\line(1,0){35}} \linethickness{0.3mm} \put(100,55){\line(1,0){10}} \linethickness{0.3mm} \put(100,45){\line(1,0){10}} \linethickness{0.3mm} \put(60,55){\line(1,0){10}} \linethickness{0.3mm} \put(60,45){\line(1,0){10}} \linethickness{0.3mm} \put(60,75){\line(1,0){50}} \put(60,35){\line(0,1){40}} \put(110,35){\line(0,1){40}} \put(60,35){\line(1,0){50}} \put(65,52){\makebox(0,0)[cc]{$\vdots$}}
\put(105,52){\makebox(0,0)[cc]{$\vdots$}}
\put(45,52){\makebox(0,0)[cc]{$\vdots$}}
\put(5,52){\makebox(0,0)[cc]{$\vdots$}}
\put(55,50){\makebox(0,0)[cc]{$=$}}
\end{picture} }\[
\raisebox{-0.8cm}[1.6cm]\figurethirteen\]
\end{proof}
\begin{cor}
$Q\mathfrak{M}_{+}Q=W^{*}(Q\mathfrak{A}_{+}Q,QXQ)$. Moreover, $QXQ$
is $Q\mathfrak{A}_{+}Q$-semicircular with variance $\eta|_{Q\mathfrak{A}_{+}Q}$.\end{cor}
\begin{proof}
$Q\mathfrak{M}_{+}Q=W^{*}(Q\mathfrak{A}_{+}Q,\{m_{i}^{*}Xm_{j}\})=W^{*}(Q\mathfrak{A}_{+}Q,QXQ)$
because of Lemma \ref{lem:XcommutesWith}. The rest follows from Lemma
\ref{lem:OpValSemicirc}.
\end{proof}
We now describe the algebra $Q\mathfrak{A}_{+}Q$. Let us write $Q_{+}=QP_{+}$
and $Q_{-}=QP_{-}$. Thus $Q_{\epsilon}$ is the unit of $Q\mathfrak{A}_{+}^{\epsilon}Q$.
The algebra $Q\mathfrak{A}_{+}^{+}Q$ is abelian. Its minimal (central)
projections are the projections $q_{v}$, $v\in\Gamma^{+}$ that we
have chosen before. For each edge $e$ of $\Gamma$, write $s(e)\in\Gamma^{+}$
and $t(e)\in\Gamma^{-}$ to denote the two ends of $e$. For each
$w\in\Gamma^{-}$ fix a minimal projection $f_{w}\in Q\mathfrak{A}_{+}^{-}Q$
in the central component associated to $w$. Next, for each edge $e$
choose a partial isometry $w_{e}\in Q\mathfrak{A}_{+}^{-}Q$ so that
$w_{e}^{*}w_{e'}=\delta_{e=e'}f_{t(e)}$ and $\sum_{s(e)=v}w_{e}w_{e}^{*}=i(q_{v})$.
Let $F=Q_{+}+\sum_{w\in\Gamma^{-}}f_{w}\leq Q$. Finally, set\begin{equation}
X_{e}=w_{e}^{*}Xq_{s(e)}+q_{s(e)}Xw_{e}\in F\mathfrak{M}_{+}F.\label{eq:Xe}\end{equation}

\begin{lem}
\label{lem:theyAreFree}$F\mathfrak{M}_{+}F=W^{*}(F\mathfrak{A}_{+}F,\{X_{e}:e\in E(\Gamma)\})$,
where $E(\Gamma)$ is the set of edges of $\Gamma$. Moreover, $X_{e}$
are free with amalgamation over $F\mathfrak{A}_{+}F$ and each $X_{e}$
is $F\mathfrak{A}_{+}F$-semicircular with variance $\eta_{e}$ given
by $\eta_{e}(a\oplus a')=(E_{1}(w_{e}a'w_{e}^{*})\oplus w_{e}^{*}i(q_{v}aq_{v})w_{e})$,
$v=s(e)$.\end{lem}
\begin{proof}
Let $F_{+}=FQ_{+}$, $F_{-}=FQ_{-}$. Then $w_{e}^{*}Xq_{s(e)}=F_{-}X_{e}F_{+}$.
We now claim that $F\mathfrak{M}_{+}F$ is generated by $F\mathfrak{A}_{+}F$
and elements $w_{e}^{*}Xq_{s(e)}$. Indeed, this follows from the
identity $\sum_{v}\sum_{s(e)=v}w_{e}w_{e}^{*}Xq_{s(e)}=Q_{-}XQ_{+}$,
which allows us to approximate an arbitrary word in $\mathfrak{A}_{+}$
and $X$ by words in $\mathfrak{A}_{+}$ and $\{X_{e}\}_{e\in E(\Gamma)}$.
We now apply Lemma \ref{lem:OpValSemicirc}(a) to $A=Q\mathfrak{A}_{+}Q$
to conclude that $\{X_{e}\}_{e\in E(\Gamma)}$ form a semicircular
family. Finally, if $e\neq e'$ and $q_{v}\in F_{+}\mathfrak{A}_{+}F_{+}$,
$f_{w}\in F_{-}\mathfrak{A}_{+}F_{-}$, then\begin{eqnarray*}
E(X_{e}aX_{e'}) & = & (w_{e}^{*}i(q_{s(e)}q_{v}q_{s(e')})w_{e'})=0\\
E(X_{e}bX_{e}') & = & (q_{s(e)}E_{1}(w_{e}f_{w}w_{e'}^{*})q_{s(e')})=0,\end{eqnarray*}
the last equality because $w_{e}f_{w}w_{e'}^{*}=\delta_{t(e)=w}w_{e}w_{e'}^{*}$
and $Tr(w_{e}w_{e}^{*}i(q_{v}))=Tr(w_{e}^{*}i(q_{v})w_{e'})=Tr(w_{e}^{*}\sum w_{f}w_{f}^{*}w_{e}')=0$
if $e\neq e'$, so that $E_{1}(wf_{w}w_{e'}^{*})=0$. Using Lemma
\ref{lem:OpValSemicirc}(b) we conclude that $X_{e}$ are free with
amalgamation over $F\mathfrak{A}_{+}F$.
\end{proof}
To simplify notation, we shall from now on write $\mathcal{A}=F\mathfrak{A}_{+}F$,
$\mathcal{A}_{\epsilon}=FQ_{\epsilon}\mathcal{A}$, $1_{\epsilon}=FQ_{\epsilon}$,
$\epsilon=\pm$. Note that the algebra $\mathcal{A}$ is abelian,
with its minimal central projections labeled by vertices of $\Gamma$.
We will write $p_{v}$ for the projection corresponding to $v\in\Gamma$
(thus $p_{v}=q_{v}$ if $v\in\Gamma^{+}$ and $p_{w}=f_{w}$ if $w\in\Gamma^{-}$).

We summarize this as:
\begin{lem}
(a) There is a trace-preserving isomorphism $(M_{0},Tr_{0})\cong(e_{0}\mathfrak{N}e_{0},Tr)$,
where $\mathfrak{N}=F\mathfrak{M}_{+}F$. (b) $\mathfrak{N}$ is isomorphic
to the amalgamated free product $*_{\mathcal{A}}\{\mathfrak{N}_{e}:e\in E(\Gamma)\}$
where $\mathfrak{N}_{e}=W^{*}(\mathcal{A},X_{e})$ and $X_{e}$ is
an $\mathcal{A}$-valued semicircular element with variance $\eta_{e}$
determined by $\eta_{e}(p_{v}\oplus0)=\lambda_{e}\delta_{s(e)=v}0\oplus p_{t(e)}$,
$\eta_{e}(0\oplus p_{w})=\lambda'_{e}\delta_{t(e)=w}p_{s(e)}\oplus0$,
where $\lambda_{e},\lambda'_{e}$ are some nonzero constant. (c) $\mathfrak{N}_{e}\cong\bigoplus_{v\in\Gamma\setminus\{s(e),t(e)\}}\mathbb{C}p_{v}\oplus W^{*}(p_{s(e)},p_{t(e)},X_{e})$.
(d) If we set $\mathcal{A}_{e}=\mathbb{C}p_{s(e)}+\mathbb{C}p_{t(e)}$,
then $X_{e}$ is $\mathcal{A}_{e}$-semicircular.\end{lem}
\begin{proof}
To prove (a), note that $e_{0}=q_{*}\leq F$. Thus $M_{0}\cong e_{0}\mathfrak{M}_{+}e_{0}=e_{0}F\mathfrak{M}_{+}Fe_{0}=e_{0}\mathfrak{N}e_{0}$
(note that in our normalization $Tr(e_{0})=1$). Part (b) is nothing
by the conclusion of Lemma \ref{lem:theyAreFree}. To see (c), we
note that $X_{e}=(p_{s(e)}+p_{t(e)})X_{e}(p_{s(e)}+p_{t(e)})$. Finally,
to see (d) we note that $\eta_{e}$ takes $\mathcal{A}_{e}$ to $\mathcal{A}_{e}$
and apply Lemma \ref{lem:OpValSemicirc}(c).\end{proof}
\begin{lem}
Let $e\in E(\Gamma)$ and let $v,w\in\Gamma$ so that $\mu(v)\leq\mu(w)$
and $\{s(e),t(e)\}=\{v,w\}$ (here $\mu$ is the Perron-Frobenius
eigenvector). Endow the algebra $\mathfrak{N}_{e}$ with the normalized
trace $\tau=\frac{1}{\mu(v)+\mu(w)}Tr$. Then there exists a trace-preserving
isomorphism between $W^{*}(\mathcal{A}_{e},X_{e})$ and the algebra
$\left(M_{2\times2}\otimes(L^{\infty}[0,1])\right)\oplus\mathbb{C}p$
where the trace on the second algebra is determined by $\tau(p)=(\mu(w)-\mu(v))/(\mu(w)+\mu(v))$.
In particular, the free dimension fdim \cite{brown-dykema-jung:fdimamalg,dykema:fdim,dykema:interpolated,dykema:finiteAlgAmalg}
of any set of generators of this algebra is given by \begin{eqnarray*}
\fdim W^{*}(\mathcal{A}_{e},X_{e}) & = & 1-\frac{1}{(\mu(w)+\mu(v))^{2}}(\mu(w)-\mu(v))^{2}\\
 & = & 1+\frac{1}{(\mu(w)+\mu(v))^{2}}\left(-\mu(v)^{2}-\mu(w)^{2}+2\mu(v)\mu(w)\right)\end{eqnarray*}
\end{lem}
\begin{proof}
Consider the algebra $\mathcal{A}_{e}$ generated by $p_{w}\mathbb{C}\oplus p_{v}\mathbb{C}$
(with unit $p_{w}+p_{v}$) and a semicircular element $Y$ which is
free from $\mathcal{A}_{e}$. By Lemma \ref{lem:OpValSemicirc}(d),
$Y$ is $\mathcal{A}_{e}$-semicircular with variance $\tau.$ Let
$X'=p_{w}Yp_{v}+p_{v}Yp_{w}$. Then $X'$ is $\mathcal{A}_{e}$-semicircular
and has the same variance (up to a scalar multiple) as $X_{e}$.

Thus\[
W^{*}(\mathcal{A}_{e},X_{e})\cong W^{*}(\mathcal{A}_{e},X').\]

Let $Z=p_{v}X'p_{w}=p_{v}Yp_{w}$. Then the distributions of $ZZ^{*}$
and $Z^{*}Z$ are free Poisson elements; moreover, $ZZ^{*}$ has support
projection $p_{v}$, while $Z^{*}Z$ has the support projection $q\leq p_{w}$
of trace $\tau(q)=\tau(p_{v})$. Let $Z=V|Z|$ be the polar decomposition
of $Z$. Thus $VV^{*}=p_{v}$, $VV^{*}=q\leq p_{w}$, so that $W^{*}(\mathbb{C}p_{v}\oplus\mathbb{C}p_{w},X_{e})\cong W^{*}(p_{v},q,V,|Z|,p_{w}-q)$.
Note that the support projection of any element in the algebra generated
by $p_{v},q,V,|Z|$ is under $p_{v}+q=1-(p_{w}-q)$ , so $W^{*}(p_{v}\mathbb{C}\oplus p_{w}\mathbb{C},X_{e})=W^{*}(p_{v},q,V,|Z|)\oplus\mathbb{C}p$,
where $p=p_{w}-q$ is a projection of trace $\frac{1}{\mu(v)+\mu(w)}(\tau(p_{w})-\tau(p_{v}))=\frac{1}{\mu(v)+\mu(w)}(\mu(w)-\mu(v))$.
On the other hand, since $p_{v}$ and $q$ are equivalent via $V$,
$W^{*}(p_{v},q,V,|Z|)$ is isomorphic to the algebra of $2\times2$
matrices over the von Neumann algebra generated by $|Z|$ in the algebra
$p_{v}W^{*}(p_{v}\mathbb{C}\oplus p_{w}\mathbb{C},Y)$. Since the
law of $|Z|$ is quarter-circular and has no atoms, the algebra generated
by $|Z|$ (with unit $p_{v}$) is isomorphic to $L^{\infty}[0,1]$.
Thus $W^{*}(p_{v}\mathbb{C}\oplus p_{w}\mathbb{C},X_{e})$ is isomorphic
to $\left(M_{2\times2}\otimes(L^{\infty}[0,1])\right)\oplus\mathbb{C}p$
as claimed.

The computation of the {}``free dimension'' $\fdim$ can be performed
using the formulas in \cite{brown-dykema-jung:fdimamalg,dykema:fdim,dykema:interpolated,dykema:finiteAlgAmalg}.
Indeed, since the algebra $M_{2\times2}\otimes(L^{\infty}[0,1])$
is hyperfinite and diffuse, its free dimension is $1$. Thus $\fdim\left[\left(M_{2\times2}\otimes(L^{\infty}[0,1])\right)\oplus\mathbb{C}p\right]=1-\tau(p)^{2}=1-\frac{1}{(\mu(w)+\mu(v))^{2}}(\mu(w)-\mu(v))^{2}$.\end{proof}
\begin{cor}
\label{cor:indentification}$M_{0}=e_{0}\mathfrak{N}e_{0}$ where
$\mathfrak{N}$ is an amalgamated free product (indexed by edges of
$\Gamma$) of type I von Neumann algebras with amalgamation over a
fixed discrete type I von Neumann subalgebra.
\end{cor}

\subsection{The finite-depth case.\label{sub:FiniteDepth}}

In the case that $\Gamma$ is infinite, the projection $F$ is infinite,
and the amalgamated free product appearing in Corollary \ref{cor:indentification}
involves an infinity of terms $\mathfrak{N}_{e}$, each semifinite.
Unfortunately, we have been unable to find existing results in the
literature to handle this case, although it is natural to conjecture
that the resulting factor is then the infinite amplification of $L(\mathbb{F}_{\infty})$.
However, in the case that $\Gamma$ is finite, the projection $F$
is finite and so $\mathfrak{N}$ is a finite factor. Moreover, each
term $\mathfrak{N}_{e}$ is also finite and is of type I. In this
case one can apply the results of and formulas in \cite{brown-dykema-jung:fdimamalg,dykema:fdim,dykema:interpolated,dykema:finiteAlgAmalg}.
\begin{lem}
\label{lem:finiteDepthIdentN}There is a trace-preserving isomorphism
$(\mathfrak{N},\frac{1}{Tr(F)}Tr)\cong(L(\mathbb{F}(s)),\tau)$ where
\begin{equation}
s=1+\frac{1}{Tr(F)^{2}}\left(-\sum_{v\in\Gamma}\mu(v)^{2}+2\sum_{e\in E(\Gamma)}\mu(t(e))\mu(s(e))\right),\label{eq:s}\end{equation}
$Tr(F)=\sum_{v\in\Gamma}\mu(v)$. \end{lem}
\begin{proof}
The algebra $\mathfrak{N}$ is an amalgamated free product over all
edges $e$ of $\Gamma$ of the algebras $\mathfrak{N}_{e}$ over the
subalgebra $\mathcal{A}$. Let $T=\sum_{v\in\Gamma}\mu(v)=Tr(F)$.

One has $\fdim(\mathcal{A})=1-T^{-2}\sum_{v\in\Gamma}\mu(v)^{2}$.
Furthermore, \begin{eqnarray*}
\fdim(\mathfrak{N}_{e}) & = & \fdim(\bigoplus_{v\in\Gamma\setminus\{t(e),s(e)\}}\mathbb{C}p_{v}\ \oplus W^{*}(\mathcal{A}_{e},X_{e}))\\
 & = & 1-\left(\sum_{v\in\Gamma\setminus\{t(e),s(e)\}}T^{-2}\mu(v)^{2}\right)-T^{-2}\left(\textrm{\raisebox{0pt}[12pt]{}}\mu(s(e))-\mu(t(e))\right)^{2}\\
 & = & 1-\left(\sum_{v\in\Gamma}T^{-2}\mu(v)^{2}\right)+T^{-2}\left(\textrm{\raisebox{0pt}[12pt]{}}\mu(s(e))^{2}+\mu(t(e))^{2}\right)\\
 &  & -T^{-2}\left(\textrm{\raisebox{0pt}[12pt]{}}\mu(t(e))^{2}+\mu(s(e))^{2}-2\mu(s(e))\mu(t(e))\right)\\
 & = & \fdim(\mathcal{A})+2T^{-2}\mu(s(e))\mu(t(e)).\end{eqnarray*}

Let $e_{0}$ an edge starting from $*$. By Lemma \ref{lem:factoriality}(c),
$\mathfrak{N}=F\mathfrak{M}_{+}F$ is a factor.

Then by \cite{brown-dykema-jung:fdimamalg,dykema:fdim,dykema:interpolated,dykema:finiteAlgAmalg},
since $\mathfrak{N}$ is a factor, it is isomorphic to $L(\mathbb{F}(s))$
where\begin{eqnarray*}
s & = & \fdim\mathfrak{N}_{e_{0}}+\sum_{e\in E(\Gamma)\setminus\{e_{0}\}}\left(\fdim\mathfrak{N}_{e}-\fdim(\mathcal{A})\right)\\
 & = & \fdim\mathcal{A}+2T^{-2}\mu(s(e_{0}))\mu(t(e_{0}))+2T^{-2}\sum_{e\in E(\Gamma)\setminus\{e_{0}\}}\mu(s(e))\mu(t(e))\\
 & = & 1+\frac{1}{T^{2}}\left(-\sum_{v\in\Gamma}\mu(v)^{2}+2\sum_{e\in E(\Gamma)}\mu(t(e))\mu(s(e))\right).\end{eqnarray*}
\end{proof}
\begin{thm}
\label{thm:main}If $\mathcal{P}$ is finite depth with global index
$I$ and $\delta>1$, then (a) $M_{0}\cong L(\mathbb{F}(1+2(\delta-1)I))$,
(b) $M_{k}\cong L(\mathbb{F}(1+2\delta^{-2k}(\delta-1)I))$.\end{thm}
\begin{proof}
By Corollary \ref{cor:indentification}, $M_{0}\cong e_{0}\mathfrak{N}e_{0}$,
where $e_{0}\in\mathfrak{N}$ is a projection of trace $1/Tr(F)=1/T$
in the notation of Lemma \ref{lem:finiteDepthIdentN}. Moreover, $\mathfrak{N}=L(\mathbb{F}(s))$
with $s$ given by \eqref{eq:s}. 

Recall that the eigenvector condition implies that $\sum_{e:s(e)=v}\mu(v)\mu(t(e))=\delta\sum\mu(v)^{2}$.
From this we find that $\sum_{e\in E(\Gamma)}\mu(s(e))\mu(t(e))=\delta\sum_{v\in\Gamma^{+}}\mu(v)^{2}$.
Let $I$ be the global index \cite{ocneanu:globalIndex,evans-kawahighashi},
$I=\sum_{v\in\Gamma^{+}}\mu(v)^{2}$. Thus $\sum_{e\in E(\Gamma)}\mu(s(e))\mu(t(e))=\delta I$. 

The eigenvector condition implies that $\sum_{v\in\Gamma}\mu(v)^{2}=2\sum_{v\in\Gamma^{+}}\mu(v)^{2}=2I$.
Hence \[
s=1+\frac{1}{T^{2}}\left(\sum_{v\in\Gamma}\mu(v)^{2}+2\sum_{e\in E(\Gamma)}\mu(s(e))\mu(t(e))\right)=1+\frac{1}{T^{2}}(-2I+2\delta I)=1+2T^{-2}(\delta-1)I.\]
The compression formula \cite{radulescu:subfact,dykema:interpolated,DVV:book}
then implies that $M_{0}\cong L(\mathbb{F}(r))$ with $r=1+T^{2}(s-1)=1+2(\delta-1)I$.
This proves (a). For $k$ even, part (b) follows from the compression
formula and Corollary \ref{cor:stableisom}. In the odd case, applying
our arguments to $\mathcal{P}^{\operatorname{op}}$ instead of $\mathcal{P}$
and using the compression formula and Corollary \ref{cor:stableisom},
we get that $M_{2k+1}(\mathcal{P}^{\operatorname{op}})\cong L(\mathbb{F}(1+2\delta^{-2k}(\delta-1)I^{\operatorname{op}})$
where $I^{\operatorname{op}}$ is the global index of $\mathcal{P}^{\operatorname{op}}$.
But $I^{\operatorname{op}}=I$ (since the principal graphs of $\mathcal{P}$
and $\mathcal{P}^{\operatorname{op}}$ share a bipartite half) and
so (b) holds.
\end{proof}

\subsection{Some examples.}

\subsubsection{Finite-dimensional Kac algebras.}

We first consider the case of the planar algebra of an $n$-dimensional
Kac algebra as considered in \cite{kodiyalam-sunder:depth2}. In this
case, the principal graph $\Gamma$ has the form\[
*\ \textrm{---}\!\bullet\!\negthinspace\!{\diagup\atop \diagdown}\!\!\!\begin{array}{c}
\bullet\\
\vdots\\
\bullet\end{array}\Bigg\}\textrm{ }(n-1)\textrm{ vertices.}\]
The Perron-Frobenius eigenvector is equal to $1$ on all even vertices
and to $\sqrt{n}$ on the single odd vertex. The associated eigenvalue
is $\delta=\sqrt{n}$. The global index is given by $I=n\cdot1=n$.
In this case, Theorem \ref{thm:main} gives:\[
M_{0}\cong L(\mathbb{F}(1+2n\sqrt{n}-2n)),\qquad M_{1}\cong L(\mathbb{F}(2\sqrt{n}-1)),\]
in accordance with the results of \cite{kodiyalam-sunder:depth2}.

\subsubsection{The inclusion $M\subset M\otimes M_{n\times n}$.}

Another example is the graph planar algebra associated to the graph
$*\!\equiv\!\equiv\!\bullet$ ($n$ edges). The associated inclusion
of II$_{1}$ factors is of the form $M\subset M\otimes M_{n\times n}$.
In this case, $\delta=n$, the eigenvector is equal to $1$ at all
vertices, and the global index is $1$. Thus $M_{0}\cong L(\mathbb{F}(2n-1))$.

\section{$\mathfrak{N_{+}}$ and random block matrices.}

We point out a connection between our description of $M_{0}=e_{0}\mathfrak{N}_{+}e_{0}=e_{0}W^{*}(\mathcal{A},\{X_{e}:e\in E(\Gamma))e_{0}$
and certain random block matrices considered in \S3.1 of \cite{guionnet-jones-shlyakhtenko1}.
To avoid the clash of notation, we will write $\bar{X}_{e}$ for what
was denoted $X_{e}$ in that paper:
\begin{prop}
\label{pro:vsPreviousPaper}The limit joint distribution of the random
matrices $\{d_{v}:v\in\Gamma\}\cup\{\bar{X}_{e}+\bar{X}_{e^{op}}^{*}:e\in E(\Gamma)\}$
as defined in \cite[{\S}3.1]{guionnet-jones-shlyakhtenko1} is the
same as that of $\{p_{v}:v\in\Gamma\}\cup\{X_{e}:e\in E(\Gamma)\}\subset W^{*}(\mathcal{A},\{X_{e}:e\in E(\Gamma)\})=\mathfrak{N}_{+}$.\end{prop}
\begin{proof}
The variables $X_{e}$ constructed in the present paper form an $\mathcal{A}$-valued
semicircular family, and so do the variables $\bar{X}_{e}+\bar{X}_{e}^{*}$
constructed in \cite{guionnet-jones-shlyakhtenko1} (see the proof
of Proposition 2).
\end{proof}
In fact, there is a clear similarity between our construction of $X_{e}$
from the variable $X$ given by equation \eqref{eq:Xe} and the construction
of $Y_{e}$ in \S3.3 of \cite{guionnet-jones-shlyakhtenko1}.

We now recall the realization of $(\mathcal{P},\wedge_{0},Tr_{0})$
constructed in \cite{guionnet-jones-shlyakhtenko1}. First, a planar
algebra $\mathcal{P}$ is realized as the subalgebra of $\mathcal{P}^{\Gamma}$
associated to its principal graph. Let us call this inclusion $\iota$.
For any $w\in\mathcal{P}$ one can then write $\iota(w)=\sum_{\rho}w_{\rho}\rho$,
where the summation takes place over all loops on $\Gamma$ starting
at an even vertex, identified with elements of $\mathcal{P}^{\Gamma}$.
If we then write $X_{\rho}=X_{e_{1}}\cdots X_{e_{n}}$ in the case
that $\rho=e_{1}\cdots e_{n}$, $e_{j}\in E(\Gamma)$, one of the
main results of \cite{guionnet-jones-shlyakhtenko1} states that the
map\[
\pi:w\mapsto\sum_{\rho}w_{\rho}\bar{X}_{\rho}\]
is a $*$-homomorphism from $(\mathcal{P},\wedge_{0},Tr_{0})$ to
the von Neumann algebra generated by $\{p_{v}:v\in\Gamma\}$ and $\{X_{e}:e\in E(\Gamma)\}$,
and that $\pi$ satisfies\[
E_{W^{*}(p_{v}:v\in\Gamma^{+})}(\pi(w))=\sum_{v\in\Gamma_{+}}Tr_{0}(w)p_{v},\]
where $E$ denotes the $Tr$-preserving conditional expectation.

We now make the observation that the map\[
\pi_{0}:w\mapsto e_{0}\pi(w)e_{0}=\sum_{\rho\textrm{ starting at }*}w_{\rho}\bar{X}_{\rho}\]
is once again a $*$-homomorphism from $(\mathcal{P},\wedge_{0},Tr_{0})$,
which this time satisfies $Tr(\pi_{0}(w))=Tr_{0}(w)$. But modulo
the identifications given by Proposition \ref{pro:vsPreviousPaper},
this is exactly the isomorphism between $(\mathcal{P},\wedge_{0},Tr_{0})$
and $e_{0}\mathfrak{N}_{+}e_{0}$.

\subsection{A short altenative proof of Theorem \ref{thm:main} for $M_{0}$.}

We note that one can \emph{deduce }Proposition \ref{pro:vsPreviousPaper}
from the fact that $\pi_{0}$ intertwines $Tr_{0}$ and $Tr$, thus
giving an alternative (and very short) proof to the fact that $M_{0}\cong e_{0}\mathfrak{N}_{+}e_{0}$,
where $\mathfrak{N_{+}}$ is \emph{defined} to be $\mathfrak{N}_{+}=W^{*}(\mathcal{A},X_{e}:e\in E(\Gamma))$.
One can then proceed as in \S\ref{sub:FiniteDepth}, thus obtaining
a shorter way of identifying the isomorphism class of $M_{0}$. We
have taken the longer route in this paper for the purpose of introducing
and exploring the extra structure of the II$_{\infty}$ factor $\mathfrak{M}_{+}$.

\subsection{Computations of moments of Jones-Wenzl projections $JW_{n}$. }

We now consider case of the $TL$ planar algebra, whose principal
graph is $A_{K}$, $K=2,3,\dots,\infty$.

Up to normalization, each simple path (i.e., a path which is a geodesic
between $*$ and the farthest point from $*$ that it reaches) of
length $2n$ corresponds to a Jones-Wenzl projection in $\mathcal{P}_{n}$.
Thus in the $A_{n}$ case we get a single Jones-Wenzl projection $JW_{k}$
for each $k=1,2,\ldots,K-1$. If we denote by $e_{j}$ the $j$-th
edge in the graph $A_{K}$, $1\leq j\leq K-1$ (in our numbering,
$e_{1}$ starts at $*$), then $JW_{k}$ corresponds to the path $e_{1}\dots e_{k}e_{k}^{o}\dots e_{1}^{o}$. 

Thus the joint law of the Jones-Wenzl projections is the same as that
of $JW_{1}=\bar{X}_{e_{1}}\bar{X}_{e_{1}}^{*}$, $JW_{2}=\bar{X}_{e_{2}}\bar{X}_{e_{1}}\bar{X}_{e_{1}}^{*}\bar{X}_{e_{2}}^{*}$,
$JW_{3}=\bar{X}_{e_{3}}\bar{X}_{e_{2}}\bar{X}_{e_{1}}\bar{X}_{e_{1}}^{*}\bar{X}_{e_{2}}^{*}\bar{X}_{e_{3}}^{*}$
and so on. 

In particular, let us use this to compute the law $\nu_{n}$ of $JW_{n}$
inductively. We have that (writing $c_{n}=P_{+}X_{e_{n}}P_{-}$ and
$Q_{n}=c_{1}\cdots c_{n}=Q_{n-1}c_{n}$):\begin{equation}
\int t^{k}d\nu_{n}(t)=Tr_{0}((JW_{n})^{k})=Tr((Q_{n}Q_{n}^{*})^{k})=Tr\left((c_{n}c_{n}^{*}\ Q_{n-1}^{*}Q_{n-1})^{k}\right).\label{eq:JW}\end{equation}
Here $Tr_{0}$ is the finite trace on $\mathcal{P},\wedge_{0}$ and
$Tr$ is the semi-finite trace on $\mathfrak{N}_{+}$. Note that both
$c_{n}c_{n}^{*}$ and $Q_{n-1}$ belong to $p_{n}\mathfrak{N}_{+}p_{n}$,
where we write $p_{n}$ for the projection at which $e_{n}$ starts.
Let $\mu_{n}=Tr(p_{n})$. Then $\mu_{n}^{-1}Tr$ is a normalized trace
on $p_{n}\mathfrak{N}_{+}p_{n}$. Moreover, we see from the random
matrix model (in which $c_{k}$ is modeled by an independent block
matrix of size $\mu_{k}N\times\mu_{k+1}N$, $N\to\infty$) that if
$u$ is a Haar unitary free from $c_{n}$ and $Q_{n-1}$ with respect
to $\mu_{n}^{-1}Tr$, then \eqref{eq:JW} does not change if we replace
$c_{n}$ by $uc_{n}$. It follows that $c_{n}c_{n}^{*}$ and $Q_{n-1}^{*}Q_{n-1}$
are free in $(p_{n}\mathfrak{N}_{+}p_{n},\mu_{n}^{-1}Tr)$, and so\[
\nu_{n}=(\textrm{law of }c_{n}c_{n}^{*}\textrm{ with respect to }\mu_{n}^{-1}Tr)\boxtimes(\textrm{law of }Q_{n-1}^{*}Q_{n-1}\textrm{ with respect to }\mu_{n}^{-1}Tr).\]
We first compute the law of $Q_{n-1}^{*}Q_{n-1}$. To this end, we
note that\[
\mu_{n}^{-1}Tr((Q_{n-1}^{*}Q_{n-1})^{k})=\mu_{n}^{-1}Tr((Q_{n-1}Q_{n-1}^{*})^{k})=\mu_{n}^{-1}\int t^{k}d\nu_{n-1}(t).\]
Thus the law of $Q_{n-1}^{*}Q_{n-1}$is given by $\mu_{n}^{-1}\nu_{n-1}+(1-\mu_{n}^{-1})\delta_{0}$.

To compute the law of $c_{n}c_{n}^{*}$ we note that it is (using
the random matrix model) the Mar\v{c}enko-Pastur distribution associated
to a $\mu_{n}N\times\mu_{n+1}N$ rectangular matrix, $N\to\infty$.
In other words, it is the Free Poisson law $\pi_{\lambda}$ with parameter
$\lambda=\mu_{n+1}/\mu_{n}$ and thus has $S$-transform\[
S_{n}(z)=\frac{1}{z+\mu_{n+1}/\mu_{n}}.\]
Thus (noting that $\mu_{1}=1$ in our normalization)\[
\nu_{n}=\pi_{\mu_{n+1}/\mu_{n}}\boxtimes(\mu_{n}^{-1}\nu_{n-1}+(1-\mu_{n}^{-1})\delta_{0}),\qquad\nu_{1}=\pi_{\mu_{2}}.\]
Let us now denote by $\Xi_{n}$ the $S$-transform of $\nu_{n}$.
Then the $S$-transform of $\mu_{n}^{-1}\nu_{n-1}+(1-\mu_{n}^{-1})\delta_{0}$
is given by\[
\Xi_{n-1}(\mu_{n}z)\frac{1+z}{\mu_{n}^{-1}+z}=\Xi_{n-1}(\mu_{n}z)\frac{\mu_{n}+\mu_{n}z}{1+\mu_{n}z}\]
Thus\begin{eqnarray*}
\Xi_{n}(z) & = & \Xi_{n-1}(\mu_{n}z)\frac{\mu_{n}(\mu_{n}+\mu_{n}z)}{(1+\mu_{n}z)(\mu_{n+1}+\mu_{n}z)},\qquad\Xi_{1}(z)=\frac{1}{z+\mu_{2}}.\end{eqnarray*}

\bibliographystyle{amsalpha}

\begin{thebibliography}{VDN92}

\bibitem[BDJ06]{brown-dykema-jung:fdimamalg}
N.~Brown, K.~Dykema, and K.~Jung, \emph{Free entropy dimension in amalgamated
  free products}, preprint math.OA/0609080, 2006.

\bibitem[Dyk93]{dykema:fdim}
K.~Dykema, \emph{Free products of hyperfinite von {Neumann} algebras and free
  dimension}, Duke Math J. \textbf{69} (1993), 97--119.

\bibitem[Dyk94]{dykema:interpolated}
\bysame, \emph{Interpolated free group factors}, Pacific J. Math. \textbf{163}
  (1994), 123--135.

\bibitem[Dyk09]{dykema:finiteAlgAmalg}
\bysame, \emph{A description of amalgamated free products of finite von neumann
  algebras over finite dimensional subalgebras}, Preprint, arXiv.org/0911.2052,
  2009.

\bibitem[EK98]{evans-kawahighashi}
D.~Evans and Y.~Kawahigashi, \emph{Quantum symmetries on operator algebras},
  Oxford Univ. Press, New York, 1998.

\bibitem[GJS08]{guionnet-jones-shlyakhtenko1}
A.~Guionnet, V.~Jones, and D.~Shlyakhtenko, \emph{Random matrices, free
  probability, planar algebras and subfactors}, To appear in Proceedings NCG,
  arXiv.org/0712.2904, 2008.

\bibitem[KSa]{kodiyalam-sunder:depth2}
V.~Kodiyalam and V.~S. Sunder, \emph{{Guionnet-Jones-Shlyakhtenko} subfactors
  associated to finite-dimensional {Kac} algebras}, Preprint,
  arXiv.org:0901.3180.

\bibitem[KSb]{kodiyalam-sunder:interpolated}
\bysame, \emph{On the {Guionnet-Jones-Shlyakhtenko} construction for graphs},
  Preprint arXiv.org:0911.2047.

\bibitem[Ocn88]{ocneanu:globalIndex}
A.~Ocneanu, \emph{Quantized group, string algebras and {Galois} theory for
  algebras}, Operator algebras and applications, London Math. Soc. Lect. Notes
  Series, vol. 136, Cambridge University Press, 1988, pp.~119--172.

\bibitem[PS03]{shlyakht-popa:universal}
S.~Popa and D.~Shlyakhtenko, \emph{Universal properties of {$L({\bf F}\sb
  \infty)$} in subfactor theory}, Acta Math. \textbf{191} (2003), no.~2,
  225--257. \MR{MR2051399 (2005b:46140)}

\bibitem[R{\u a}d94]{radulescu:subfact}
F.~R{\u a}dulescu, \emph{Random matrices, amalgamated free products and
  subfactors of the von {Neumann} algebra of a free group, of noninteger
  index}, Invent. math. \textbf{115} (1994), 347--389.

\bibitem[Shl98]{shlyakht:amalg}
D.~Shlyakhtenko, \emph{Some applications of freeness with amalgamation}, J.
  reine angew. Math. \textbf{500} (1998), 191--212.

\bibitem[Shl99]{shlyakht:semicirc}
\bysame, \emph{{$A$}-valued semicircular systems}, J. Func. Anal \textbf{166}
  (1999), 1--47.

\bibitem[Spe98]{speicher:thesis}
R.~Speicher, \emph{Combinatorial theory of the free product with amalgamation
  and operator-valued free probability theory}, Mem. Amer. Math. Soc.
  \textbf{132} (1998), x+88.

\bibitem[VDN92]{DVV:book}
D.-V. Voiculescu, K.~Dykema, and A.~Nica, \emph{Free random variables}, CRM
  monograph series, vol.~1, American Mathematical Society, 1992.

\end{thebibliography}

\providecommand{\bysame}{\leavevmode\hbox to3em{\hrulefill}\thinspace}
\providecommand{\MR}{\relax\ifhmode\unskip\space\fi MR }
\providecommand{\MRhref}[2]{%
  \href{http://www.ams.org/mathscinet-getitem?mr=#1}{#2}
}
\providecommand{\href}[2]{#2}

\end{document}